\documentclass[letterpaper, final]{IEEEtran}
\usepackage{epsfig} % for postscript graphics files
\usepackage{mathptmx} % assumes new font selection scheme installed
\usepackage{times} % assumes new font selection scheme installed
\usepackage{amsmath} % assumes amsmath package installed
\usepackage{amssymb}  % assumes amsmath package installed
\usepackage{amsthm}

\usepackage{url}

\usepackage{calc}
\usepackage{pdfsync}
\usepackage{enumerate}
\usepackage[singlelinecheck=off,font=footnotesize]{caption} 
\usepackage[usenames,dvipsnames]{xcolor}
\usepackage{tikz}
\usetikzlibrary{arrows}
\usetikzlibrary{graphs,decorations.pathmorphing,decorations.markings}
\usetikzlibrary{calc}

\pgfmathsetmacro\weight{1/2}
\pgfmathsetmacro\third{1/3}
\pgfmathsetmacro\twothirds{2/3}

\tikzset{degil/.style={
            decoration={markings,
            mark= at position 0.5 with {
                  \node[transform shape] (tempnode) {$/$};
                  }
              },
              postaction={decorate}
}
}

\newtheorem{Satz}{Theorem} %for model ~(\ref{OekosystemModell}
\newtheorem{Lemm}{Lemma} %for model ~(\ref{OekosystemModell}
\newtheorem{Def}{Definition} 

\newtheorem{remark}{Remark}

\newtheorem  {proposition}[Satz] {Proposition}
\newtheorem  {examp}     {Example}

\newtheorem{Ass}{Assumption}

\newcommand \N   {\mathbb{N}}
\newcommand \R   {\mathbb{R}}

\newcommand \K   {\mathcal{K}}
\newcommand \Kinf{\mathcal{K_\infty}}

\newcommand \KL  {\mathcal{KL}}
\newcommand \LL  {\mathcal{L}}

\newcommand \Uc   {\mathcal{U}}

\newcommand \Sc   {\mathcal{S}}

\newcommand \srs   {\ \ \Rightarrow\ \ }

\newcommand \Iff   {\Leftrightarrow}

\newcommand{\lel}{\left\langle}
\newcommand{\rir}{\right\rangle}

\newcommand \eps {\varepsilon}

\newcommand{\midset}{\;:\;}

\newcounter{syscounter}
\newenvironment{sysnum}{\begin{list}{($\Sigma{\arabic{syscounter}}$)}%
{\settowidth{\labelwidth}{($\Sigma4$)}
\settowidth{\leftmargin}{($\Sigma4$)~}%
\usecounter{syscounter}}}
{\end{list}}

\begin{document}

\title{%\LARGE \bf
%Restatements of input-to-state stability in infinite dimensions:\\ what goes wrong?
Characterizations of input-to-state stability \\for infinite-dimensional systems
}
\author{Andrii Mironchenko and Fabian Wirth
\thanks{
This research has been supported by the DFG under grant "Input-to-state stability
and stabilization of distributed parameter systems" (Wi1458/13-1).
%This research is funded by the DFG under grant Wi 1458/10.
This paper is a substantially revised and expanded version of conference
papers presented at the 22nd International Symposium on Mathematical
Theory of Networks and Systems (MTNS 2016) \cite{MiW16a} and the 56th IEEE
Conference on Decision and Control (CDC2017) \cite{MiW17d}.
}
\thanks{A. Mironchenko is with 
Faculty of Computer Science and Mathematics, University of Passau,
94030 Passau, Germany.
Email: andrii.mironchenko@uni-passau.de. Corresponding author.
}
\thanks{F. Wirth is with 
Faculty of Computer Science and Mathematics, University of Passau,
94030 Passau, Germany.
Email: fabian.lastname@uni-passau.de.
}
}

\maketitle
%\thispagestyle{empty}
%\pagestyle{empty}
%%%%%%%%%%%%%%%%%%%%%%%%%%%%%%%%%%%%%%%%%%%%%%%%%%%%%%%%%%%%%%%%%%%%%%%%%%%%%%%%
\begin{abstract}
We prove characterizations of input-to-state stability (ISS) for a large class of
infinite-dimensional control systems, including 
some classes of evolution equations over Banach spaces, time-delay systems, ordinary differential equations (ODE),
switched systems. These characterizations generalize well-known criteria
of ISS, proved by Sontag and Wang for ODE systems. For the special case of
differential equations in Banach spaces we prove even broader criteria for
ISS and apply these results to show that (under some mild restrictions)
the existence of a non-coercive ISS Lyapunov functions implies ISS.
We introduce the new notion of strong ISS which is equivalent to ISS in
the ODE case, but which is strictly weaker than ISS in the
infinite-dimensional setting and prove several criteria for the sISS property.
At the same time, we show by means of counterexamples, that many
characterizations, which are valid in the ODE case, are not true for general infinite-dimensional systems.
\end{abstract}

%\textbf{Keywords}: input-to-state stability, nonlinear systems, infinite-dimensional systems.


\section{Introduction}

For ordinary differential equations, the concept of input-to-state
stability (ISS) was introduced in \cite{Son89}. The corresponding theory is now well developed and has a firm theoretical basis. Several powerful tools for the investigation of ISS are available and a
multitude of applications have been developed in nonlinear control theory, in particular, to robust stabilization of nonlinear systems \cite{FrK08}, design of nonlinear observers \cite{ArK01}, analysis of large-scale networks \cite{JTP94, DRW07, DRW10}, etc.

The success of ISS theory of ordinary differential equations and the need of proper tools for robust stability analysis of partial differential equations motivated the development of ISS theory in infinite-dimensional setting \cite{DaM13, MiI16, MiI15b, MaP11, JLR08, KaK16b, Mir16, JNP16}.

Characterizations of ISS in terms of other stability properties
\cite{SoW95, SoW96} are among the central theoretical results in ISS theory
of finite-dimensional systems. In \cite{SoW95} Sontag and Wang have shown
that ISS is equivalent to the existence of a smooth ISS Lyapunov function and in
\cite{SoW96} the same authors proved a so-called ISS superposition
theorem, saying that ISS is equivalent to the combination of an asymptotic gain (AG) property of the system with inputs together with global/local stability (GS/LS), and even local stability of the undisturbed system (0-LS):
\begin{equation}
  \mathrm{AG} \wedge \mathrm{GS} \;\Leftrightarrow\; \mathrm{AG} \wedge \mathrm{LS} \;\Leftrightarrow\; \mathrm{AG} \wedge \text{0-LS} \;\Leftrightarrow\; \mathrm{ISS},
\label{eq:Equivalences_ODEs}   
\end{equation}
 see \cite{SoW95,SoW96}.

These theorems greatly simplify the proofs of other fundamental results, such as small-gain theorems \cite{DRW07},
and are useful for analysis of other classes of systems, such as
time-delay systems in the Lyapunov-Razumikhin framework \cite{Tee98},
\cite{DKM12} as well as hybrid systems \cite{DaK13} to name a few examples.

The significance of these characterizations of ISS makes it strongly desirable to extend the results to infinite-dimensional systems.
In the recent paper \cite{Mir16} it was shown that uniform
asymptotic stability at zero, local ISS and the existence of a LISS Lyapunov
function are equivalent properties for a system of the form
\begin{equation}
\label{InfiniteDim_Intro}
\dot{x}(t)=Ax(t)+f(x(t),u(t)), \quad x(t) \in X,\ u(t) \in U,
\end{equation}
provided the right hand side has some sort of uniform continuity with
respect to $u$. Here $X$ is a Banach space, $U$ is a linear normed space,
$A$ is the generator of a $C_0$-semigroup $\{T(t),\; t \geq 0 \}$ and $f:X
\times U \to X$ is sufficiently regular. 
It was also demonstrated by means of a counterexample, that without this 
additional uniformity this characterization does not hold.

%Many classes of evolution equations, such as parabolic and hyperbolic partial differential equations, can be written in the form \eqref{InfiniteDim}: \cite{Hen81}, \cite{CaH98}, \cite{JaZ12}.

In addition, in \cite{Mir16} a system of the form \eqref{InfiniteDim_Intro} was
constructed, which is locally ISS (LISS), uniformly globally
asymptotically stable for a zero input (0-UGAS), globally stable (GS) and
which has an asymptotic gain (AG), but which is not ISS, which strikingly contrasts with the ODE case, see \eqref{eq:Equivalences_ODEs}.

 This naturally leads to a set of challenging questions: 
which combinations of properties considered in \cite{SoW96} are equivalent to ISS for infinite-dimensional systems? 
Is it possible to generalize all characterizations of ISS from \cite{SoW96} to the general infinite-dimensional setting, and under which conditions? Can one classify the properties, which are not equivalent to ISS in a natural way? 
Is it possible to introduce a reasonable ISS-like property which will be equivalent to ISS
in finite dimensions, but weaker than ISS for general systems
\eqref{InfiniteDim_Intro}?

In this paper, we are going to answer these questions and obtain a broad picture of relationships between stability properties 
for a large class of infinite-dimensional control systems, encompassing ODEs, differential equations in Banach spaces, time-delay systems, switched systems, etc.

In view of the examples in \cite{Mir16}, we know that a "naive"
generalization of the equivalences \eqref{eq:Equivalences_ODEs} is not
possible. These preliminary studies reveal a lack of uniformity with
respect to the state in the definition of AG and other properties. 
In finite dimensions, uniform and non-uniform notions are frequently equivalent due to local compactness of
the state space. In infinite dimensions, this uniformity becomes a requirement.  A
further difficulty we encounter in infinite-dimensional systems is that,
in contrast to the ODE case, forward completeness or global
asymptotic stability do not guarantee the
boundedness of reachability sets on finite time intervals. This is shown
in the sequel by means of a counterexample.

In order to overcome these difficulties, we introduce several novel
stability notions, which naturally extend the concepts of limit property
and of asymptotic gain.
 Namely: the uniform limit property (ULIM), the
strong limit property (sLIM) as well as the strong asymptotic gain property
(sAG). 

We say that a system has the uniform limit property (ULIM) if there exists a continuous, positive definite and increasing function
    $\gamma$ so that for any $\eps>0$ and for every $r>0$ there exists a $\tau = \tau(\eps,r)$ such that 
for all $x$ with $\|x\|_X \leq r$ and all $u\in\Uc$ there is a $t\leq
\tau$ such that 
\begin{eqnarray*}
\|\phi(t,x,u)\|_X \leq \eps + \gamma(\|u\|_{\Uc}).
%\label{eq:ULIM_ISS_section}
\end{eqnarray*}
The ULIM property with zero gain, i.e. $\gamma \equiv 0$, is also called uniform weak attractivity \cite{MiW17a}.
%It was used in \cite{Mir17a} to characterize practical uniform global asymptotic stability

\textit{It turns out that ULIM is the key to obtaining generalizations of
the characterizations of ISS, see
Theorems~\ref{thm:MainResult_Characterization_ISS} and \ref{thm:UAG_equals_ULIM_plus_LS}.} 
For a class of evolution equations with Lipschitz continuous nonlinearities we obtain in Theorem~\ref{thm:MainResult_Characterization_ISS_EQ_Banach_Spaces} additional characterizations in terms of ULIM together
with local stability of the undisturbed system.
In turn, with the help of Theorem~\ref{thm:MainResult_Characterization_ISS_EQ_Banach_Spaces} and recent results on non-coercive Lyapunov functions \cite{MiW17a} we 
show in Theorem~\ref{thm:ncISS_LF_sufficient_condition} that (under certain restrictions) \textit{existence of a non-coercive ISS Lyapunov function implies ISS}.

Using the notions of sLIM and sAG
we can characterize what we call strong ISS (sISS). For linear systems
without inputs, this concept reduces to strong stability of the semigroup
$T$, whereas ISS for linear systems without inputs corresponds to
exponential stability of $T$. In order to characterize strong ISS we
introduce the strong asymptotic gain (sAG) property, which is weaker than the
uniform asymptotic gain (UAG) property, and prove that \textit{strong ISS is equivalent to global stability together with sAG}, see Theorem~\ref{wISS_equals_sAG_GS}.

In the finite-dimensional case, we show in Proposition~\ref{prop:ULIM_equals_LIM_in_finite_dimensions}
 that the sLIM and ULIM properties
are equivalent to the usual limit property introduced in
\cite{SoW96}. This proof relies in an essential manner on on tools already
developed in \cite{SoW96}.
On the other hand, ULIM is strictly stronger than sLIM or LIM already for
linear infinite-dimensional systems. In particular, we recover all
characterizations of ISS for ODEs from \cite{SoW96} as a special case of our
results.

As argued above, for \eqref{InfiniteDim_Intro} ISS is no longer equivalent
to combinations of notions which are not fully uniform - like AG $\wedge$
GS or AG $\wedge$ 0-UGAS. By means of counterexamples, we show that these
combinations are no longer equivalent to each other. Instead, they can
be classified into several groups, according to the type and grade of
uniformity.

The manuscript is structured as follows.
In Section~\ref{sec:Prelim} we introduce the main concepts which will be used throughout the paper.
In Section~\ref{sec:Motivation_MainResult} we motivate the topic of the
paper in more precise terms, state the \textit{main results of the paper (Theorems~\ref{thm:MainResult_Characterization_ISS},~\ref{thm:UAG_equals_ULIM_plus_LS})} and
explain the way it is proved. In the same section we apply our main results to show that for a broad class of evolution equations the existence of a non-coercive ISS Lyapunov function implies ISS.
The subsequent sections contain the proof of the main result. First, in Section~\ref{ISS_equals_UAG} we prove characterizations of ISS for general infinite-dimensional systems in terms of uniform limit and uniform attraction properties.
In Section~\ref{sec:WeakISS} a concept of strong ISS is introduced and characterized in terms of strong limit and strong asymptotic gain properties.
In Section~\ref{sec:Counterexamples} we construct four counterexamples,
which clarify the interrelations between the different stability notions as
well as some of the difficulties and obstacles arising in infinite-dimensional ISS theory.

The results in this paper are complementary to recently submitted papers on Lyapunov characterizations of ISS \cite{MiW17c} and on characterizations of UGAS for infinite-dimensional systems with disturbances by means of non-coercive Lyapunov functions \cite{MiW17a}. 
Although the results in this paper are almost independent from those in \cite{MiW17a,MiW17c} (apart from \cite[Lemma 2.12]{MiW17a}), we subsume the main results of \cite{MiW17a,MiW17c,Mir16} into Theorem~\ref{thm:MainResult_Characterization_ISS}
and Proposition~\ref{Main_Prop_Undisturbed_Systems} from this paper in order to give a reader a broader perspective on characterizations of ISS.

%In this paper we exploit the following notation.
%
%For linear normed spaces $X,Y$ let $L(X,Y)$ be the spaces of bounded linear operators from $X$ to $Y$ and $L(X):=L(X,X)$. A norm in these spaces we denote by $\| \cdot \|$. 
%
%By $C(X,Y)$ we denote the space of continuous functions from $X$ to $Y$, $C(X):=C(X,X)$ and by 
%$PC(X,Y)$ the space of piecewise right-continuous functions from $X$ to $Y$. Both are equipped with the standard $\sup$-norm.

\section{Preliminaries}
\label{sec:Prelim}

We define the concept of (time-invariant) system in the following way:

%\index{control system}
\begin{Def}
\label{Steurungssystem}
Consider the triple $\Sigma=(X,\Uc,\phi)$ consisting of 
%\index{state space}
%\index{space of input values}
%\index{input space}
\begin{enumerate}[(i)]  
    \item A normed linear space $(X,\|\cdot\|_X)$, called the {state space}, endowed with the norm $\|\cdot\|_X$.
    \item A set of input values $U$, which is a nonempty subset of a certain normed linear space.
    \item A {space of inputs} $\Uc \subset \{f:\R_+ \to U\}$          
endowed with a norm $\|\cdot\|_{\Uc}$ which
          satisfies the following two axioms:
                    
\textit{The axiom of shift invariance} states that for all $u \in \Uc$ and all $\tau\geq0$ the time
shift $u(\cdot + \tau)\in\Uc$ with \mbox{$\|u\|_\Uc \geq \|u(\cdot + \tau)\|_\Uc$}.
%the latter is used in the proof of Lemma 6.

\textit{The axiom of concatenation} is defined by the requirement that for all $u_1,u_2 \in \Uc$ and for all $t>0$ the concatenation of $u_1$ and $u_2$ at time $t$
\begin{equation}
u(\tau) := 
\begin{cases}
u_1(\tau), & \text{ if } \tau \in [0,t], \\ 
u_2(\tau-t),  & \text{ otherwise},
\end{cases}
\label{eq:Composed_Input}
\end{equation}
belongs to $\Uc$.

    \item A map $\phi:\R_+ \times X \times \Uc \to X$ (called transition map), defined over a certain subset of $\R_+ \times X \times \Uc$. 
\end{enumerate}
The triple $\Sigma$ is called a (forward complete) dynamical system, if the following properties hold:

\begin{sysnum}
    \item\label{axiom:Identity} The identity property: for every $(x,u) \in X \times \Uc$
          it holds that $\phi(0, x,u)=x$.
    \item Causality: for every $(t,x,u) \in \R_+ \times X \times
          \Uc$, for every $\tilde{u} \in \Uc$, such that $u(s) =
          \tilde{u}(s)$ for all $s \in [0,t]$ it holds that $\phi(t,x,u) = \phi(t,x,\tilde{u})$.
    \item \label{axiom:Continuity} Continuity: for each $(x,u) \in X \times \Uc$ the map $t \mapsto \phi(t,x,u)$ is continuous.
  %\item \label{axiomsyscont}
          %Continuity with respect to initial conditions: for all $d \in
          %\Uc, T\geq 0$ the map
          %$\phi(\cdot,\cdot,d): X \to C([0,T],X)$, $x \mapsto
          %\phi_{[0,T]}(\cdot,x,d)$ is continuous with respect to the
          %supremum norm on $C([0,T],X)$.
    %\item \label{axiom:Lipschitz} Lipschitz continuity w.r.t. ininital conditions on compact intervals: For any $\tau>0$ and any $R>0$ there exists $L>0$ so that for any $x,y \in X$: $\|x\|_X \leq R$, $\|y\|_X \leq R$, for all $t \in [0,\tau]$ and for all $d \in \Uc$ it holds that 
%\begin{eqnarray}
%\|\phi(t,x,d) - \phi(t,y,d) \|_X \leq L \|x-y\|_X.
%\label{eq:Flow_is_Lipschitz}
%\end{eqnarray}                
%\index{semigroup property}
%     \item Semigroup property: for all $t,h \geq 0$, for all $x \in X$, $d_1,d_2 \in \Uc$ it follows
% $\phi(h,\phi(t,x,d_1),d_2)=\phi(t+h,x,d)$, where $d$ is defined as in \eqref{eq:Composed_Input}.    
        \item \label{axiom:Cocycle} The cocycle property: for all $t,h \geq 0$, for all
                  $x \in X$, $u \in \Uc$ we have
$\phi(h,\phi(t,x,u),u(t+\cdot))=\phi(t+h,x,u)$, whenever the left or the
right hand side of this equality is defined.
\end{sysnum}

We say that a control system is forward complete, if in addition to the above axioms it holds that
\begin{itemize}
   \item[(FC)]\label{axiom:FC} Forward completeness: for every $(x,u) \in X \times \Uc$ and
          for all $t \geq 0$ the value $\phi(t,x,u) \in X$ is well-defined.
\end{itemize}

\end{Def}
This class of systems encompasses control systems generated by ordinary
differential equations (ODEs), switched systems, time-delay systems,
evolution partial differential equations (PDEs), abstract differential
equations in Banach spaces and many others. 
From now on, we consider only forward-complete control systems.

\begin{remark}
Note however, that not all important systems are covered by our definitions. In particular, 
the input space $C(\R_+,U)$ of continuous $U$-valued functions does not satisfy the axiom of concatenation.
This, however, should not be a big restriction, since already piecewise continuous and $L_p$ inputs, which are used in control theory much more frequently than continuous ones, satisfy the axiom of concatenation.

Some authors consider more general concepts, in which the systems fail to satisfy the cocycle property, see e.g. \cite{KaJ11b}.
\end{remark}

%{\color{red} Do we want to tell something on terminology: why do we call cocycle as a cocycle, that other guys call it semigroup property, etc.
%}

We single out two particular cases which will be of interest.

ISS of the following class of semi-linear infinite-dimensional systems has been
studied in \cite{Mir16}. Let $A$ be the generator of a strongly continuous
semigroup $T$ of bounded linear operators on $X$ and let $f:X\times U \to X$. Consider the system
\begin{equation}
\label{InfiniteDim}
\dot{x}(t)=Ax(t)+f(x(t),u(t)), \quad u(t) \in U,
\end{equation}
where $x(0)\in X$.

We study mild solutions of \eqref{InfiniteDim}, i.e. solutions $x:[0,\tau] \to X$ of the integral equation
\begin{align}
\label{InfiniteDim_Integral_Form}
x(t)=T(t)x(0) + \int_0^t T(t-s)f(x(s),u(s))ds,
\end{align}
belonging to the space of continuous functions $C([0,\tau],X)$ for some $\tau>0$.

In the sequel we assume that the state space $X$ is a Banach space, the set of input values $U$ is a normed
linear space and that the input functions belong to the space
$\Uc:=PC(\R_+,U)$ of globally bounded, piecewise continuous functions $u:\R_+ \to U$, which
are right continuous. The norm of $u \in \Uc$ is given by
$\|u\|_{\Uc}:=\sup_{t \geq 0}\|u(t)\|_U$. 

\begin{remark}
Note that there are interesting infinite-dimensional systems, which are not covered by the class of systems \eqref{InfiniteDim}.
In particular, boundary control systems can be described by control systems with unbounded input operators \cite{JaZ12},
which is not covered by \eqref{InfiniteDim}. Some highly nonlinear systems
(even without inputs) as e.g. the porous medium equation \cite{Vaz07}, the
nonlinear KdV equation \cite{BaK00} or nonlinear Fokker-Planck equations \cite{Fra05} are not covered by \eqref{InfiniteDim}, and should be modeled using methods of nonlinear semigroup theory \cite{Bar10}.
\end{remark}

\textbf{Notation:} We use the following notation. The nonnegative reals are $\R_+:=[0,\infty)$. The open ball of
radius $r$ around $0$ in $X$ is denoted by $B_r:=\{x \in X: \|x\|_X <
r\}$. Similarly, \mbox{$B_{r,\Uc}:=\{u \in \Uc: \|u\|_\Uc < r\}$}. By $\mathop{\overline{\lim}}$ we denote the superior limit.
For any normed linear space $\cal L$, for any $S \subset \cal L$ we denote
the closure
$\overline{S}:=\{f\in {\cal L}: \exists \{f_k\}\subset S \mbox{ s.t. } \|f_k-f\|_{\cal L}\to 0,\ k\to\infty\}$.

For the formulation of stability properties the following classes of comparison functions are useful:
%\begin{equation*}
%\begin{array}{ll}
%{\K} &:= \left\{\gamma:\R_+ \to \R_+ \left|\ \right. \gamma\mbox{ is continuous and strictly increasing, }\gamma(0)=0\right\}\\
%{\K_{\infty}}&:=\left\{\gamma\in\K\left|\ \gamma\mbox{ is unbounded}\right.\right\}\\
%{\LL}&:=\left\{\gamma:\R_+ \to \R_+ \left|\ \gamma\mbox{ is continuous and decreasing with}\right.
 %\lim\limits_{t\rightarrow\infty}\gamma(t)=0 \right\}\\
%{\KL} &:= \left\{\beta: \R_+^2 \to \R_+ \left|\ \right. \beta(\cdot,t)\in{\K},\ \forall t \geq 0,\  \beta(r,\cdot)\in {\LL},\ \forall r >0\right\}
%\end{array}
%\end{equation*}
\begin{equation*}
\begin{array}{ll}
{\K} &:= \left\{\gamma:\R_+\rightarrow\R_+\left|\ \gamma\mbox{ is continuous, strictly} \right. \right. \\
&\phantom{aaaaaaaaaaaaaaaaaaa}\left. \mbox{ increasing and } \gamma(0)=0 \right\}, \\
{\K_{\infty}}&:=\left\{\gamma\in\K\left|\ \gamma\mbox{ is unbounded}\right.\right\},\\
{\LL}&:=\left\{\gamma:\R_+\rightarrow\R_+\left|\ \gamma\mbox{ is continuous and strictly}\right.\right.\\
&\phantom{aaaaaaaaaaaaaaaa} \text{decreasing with } \lim\limits_{t\rightarrow\infty}\gamma(t)=0\},\\
{\KL} &:= \left\{\beta:\R_+\times\R_+\rightarrow\R_+\left|\ \beta \mbox{ is continuous,}\right.\right.\\
&\phantom{aaaaaa}\left.\beta(\cdot,t)\in{\K},\ \beta(r,\cdot)\in {\LL},\ \forall t\geq 0,\ \forall r >0\right\}. \\
\end{array}
\end{equation*}

For system \eqref{InfiniteDim}, we use the following assumption concerning the nonlinearity $f$.
\begin{Ass}
\label{Assumption1} We assume that:
\begin{enumerate}[(i)]  
    \item $f:X \times U \to X$ is Lipschitz continuous on bounded
subsets of $X$, uniformly with respect to the second argument, i.e.  for
all $C>0$, there exists a $L_f(C)>0$, such that for all $ x,y \in B_C $ and for all $v \in U$, it holds that
\begin{eqnarray}
\|f(x,v)-f(y,v)\|_X \leq L_f(C) \|x-y\|_X.
\label{eq:Lipschitz}
\end{eqnarray}
    \item $f(x,\cdot)$ is continuous for all $x \in X$
    and $f(0,0)=0$.
\end{enumerate}
\end{Ass}

Since $\Uc=PC(\R_+,U)$, Assumption~\ref{Assumption1} ensures that mild
solutions of initial value problems of \eqref{InfiniteDim} exist and are
unique, according to
% a variation of a classical existence and uniqueness
%theorem 
\cite[Proposition 4.3.3]{CaH98}. For system \eqref{InfiniteDim}
forward completeness is a further assumption. The conditions
($\Sigma$\ref{axiom:Identity})-($\Sigma$\ref{axiom:Cocycle}) are satisfied
by construction.

The second case of interest are finite-dimensional systems. Let $X=\R^n$,
$U=\R^m$ and $\Uc:=L_{\infty}(\R_+,U)$ (the space of globally
essentially bounded $U$-valued functions endowed with the essential supremum
norm). For $f: X\times U \to X$ consider the system
\begin{eqnarray}
\dot{x} = f(x,u),
\label{eq:ODE_System}
\end{eqnarray}
and define by $\phi(t,y,v)$ the solution of \eqref{eq:ODE_System} at time $t$ subject to initial condition
$x(0):=y$ and $u:=v$.
We assume that $f$ is continuous and locally Lipschitz continuous in $x$
uniformly in $u$. With this assumption and the additional assumption of
forward completeness classical Carath{\'e}odory theory implies  
($\Sigma$\ref{axiom:Identity})-($\Sigma$\ref{axiom:Cocycle}). We will
sometimes briefly speak of ODE systems, when referring to \eqref{eq:ODE_System}.

We start with some basic definitions. 
Without loss of generality we
restrict our analysis to fixed points of the form $(0,0) \in X \times
\Uc$, so that we tacitly assume that the zero input is an element of $\Uc$. 
\begin{Def}
\label{Assumption2}
Consider a system $\Sigma=(X,\Uc,\phi)$.
We call $0 \in X$ an equilibrium point (of the undisturbed system)  if
$\phi(t,0,0) = 0$ for all $t \geq 0$.
\end{Def}

%{\color{red} 
%Some problems may arise with terminology: in \cite{MiW17a} we use the notion of a robust equilibrium in another sense
%(as in \cite{KaJ11b}).
%}

For describing the behavior of solutions near the equilibrium the following notion is of importance
\begin{Def}
\label{def:RobustEquilibrium_Undisturbed}
Consider a system $\Sigma=(X,\Uc,\phi)$ with equilibrium point $0\in X$.
We say that 
$\phi$ is continuous at the equilibrium if   
%
%the equilibrium point is robust if
% \begin{itemize}
%     \item $\phi(t,0,0) = 0$ for all $t \geq 0$
%    \item 
for every $\eps >0$ and for any $h>0$ there exists a $\delta =
          \delta (\eps,h)>0$, so that 
%\vspace*{-0.7cm}
%\end{itemize}
\begin{eqnarray}
 t\in[0,h],\ \|x\|_X \leq \delta,\ \|u\|_{\Uc} \leq \delta \; \Rightarrow \;  \|\phi(t,x,u)\|_X \leq \eps.
\label{eq:RobEqPoint}
\end{eqnarray}
In this case we will also say that $\Sigma$ has the CEP property.
\end{Def}

Even nonuniformly globally asymptotically stable systems do not always have uniform bounds for their reachability sets on finite
intervals (see Example~\ref{0-GAS_but_not_GS}). Systems exhibiting such bounds
deserve a special name.
\begin{Def}
\label{Assumption3}
We say that $\Sigma=(X,\Uc,\phi)$ has bounded reachability sets (BRS), if for any $C>0$ and any $\tau>0$ it holds that 
\[
\sup\big\{
\|\phi(t,x,u)\|_X \midset \|x\|_X\leq C,\ \|u\|_{\Uc} \leq C,\ t \in [0,\tau]\big\} < \infty.
\]
\end{Def}
%Robustness of equilibrium points and boundedness of reachability sets are basic prerequisites for uniform global stability. 

%We make two more assumptions:
%
%\begin{Ass}
%\label{Assumption2}
%We assume that $0 \in X$ is an equilibrium point of an undisturbed system \eqref{InfiniteDim}, that is:
%$\phi(t,0,0) = 0$ for all $t \geq 0$. In other words, we require that $f(0,0)=0$.
%\end{Ass}
%
%\begin{Ass}
%\label{Assumption3}
%We assume that \eqref{InfiniteDim} is forward complete and has bounded reachability sets (BRS), that is for any $C>0$ and any $\tau>0$ it holds that 
%\[
%\sup_{\|x\|_X\leq C,\ \|u\|_{\Uc} \leq C,\ t \in [0,\tau]}\|\phi(t,x,u)\|_X < \infty.
%\]
%\end{Ass}

We continue with the list of stability notions, which will be used in the sequel. 
Several of these were already
  introduced in \cite{SoW96} while others appear here for the first time as they only become relevant in the infinite-dimensional case. In the finite-dimensional case, these new notions coincide with the classic ones. We discuss this issue in Section~\ref{sec:ODE_LIM_ULIM}.

\subsection{Stability notions for undisturbed systems}

We start with systems without inputs.

\begin{Def}
\label{Stab_Notions_Undisturbed_Systems}
System $\Sigma=(X,\Uc,\phi)$ is called
\begin{itemize}%[leftmargin=*]
\item[(i)] {\it uniformly stable at zero (0-ULS)}, if there exists a \mbox{$\sigma \in \Kinf$} and $r>0$ so that 
 \begin{equation}
 \label{eq:1}
\|\phi(t,x,0)\|_X \leq \sigma(\|x\|_X) \quad \forall x \in \overline{B_r}\ \forall t \geq 0.              
 \end{equation}
%\item {\it uniformly stable at zero (0-ULS)}, if  for all $ \eps >0$ there exists a 
 %$\delta>0$ so that 
 %\begin{equation}
 %\label{eq:1}
 %\|\phi(t,x,0)\|_X < \eps \quad \forall x \in B_{\delta}\ \forall t \geq 0.     
 %\end{equation}
% $\forall x \in B_{\delta}$
% it follows $\|\phi(t,x,0)\|_X < \eps$.

    \item[(ii)] {\it uniformly globally stable at zero (0-UGS)}, if there exists a $
          \sigma \in \Kinf$ so that 
          \begin{equation}
              \label{eq:3}
\|\phi(t,x,0)\|_X \leq \sigma(\|x\|_X) \quad\forall x \in X\ \forall t \geq 0.              
          \end{equation}
%
    %\item {\it practically uniformly globally stable at zero (0-pUGS)}, if there exist $
          %\sigma \in \Kinf$ and $c>0$ so that 
          %\begin{equation}
              %\label{eq:0pUGS}
%\|\phi(t,x,0)\|_X \leq \sigma(\|x\|_X) + c \quad \forall x \in X\  \forall t \geq 0.              
          %\end{equation}    
    \item[(iii)] {\it globally attractive at zero (0-GATT)}, if 
\begin{equation}
\label{GATT_LIM} 
\lim_{t \to \infty} \left\| \phi(t,x,0) \right\|_X = 0\quad \forall x\in X.
\end{equation}
    
    \item[(iv)] a system with the {\it limit property at zero (0-LIM)}, if 
\[
\inf_{t \geq 0} \|\phi(t,x,0)\|_X = 0\quad \forall x\in X.
\]
    
    \item[(v)] {\it uniformly globally attractive at zero (0-UGATT)}, if
          for all $\eps, \delta >0$ there is a
          $\tau=\tau(\eps,\delta) < \infty$ %($a$ stands for
                                %'attraction'):
such that 
\begin{equation}
\label{UnifGATT}
 t \geq \tau,\ x \in \overline{B_{\delta}} \quad \Rightarrow \quad \|\phi(t,x,0)\|_X \leq \eps.
\end{equation}

  \item[(vi)] {\it globally asymptotically stable at zero} (0-GAS), if $\Sigma$ is 0-ULS and 0-GATT. 

    \item[(vii)] {\it asymptotically stable at zero uniformly with
            respect to the state} (0-UAS), if there exists a $ \beta \in \KL$
          and $r>0$, such that 
\begin{equation}
\label{UniStabAbschaetzung}
\left\| \phi(t,x,0) \right\|_{X} \leq  \beta(\left\| x \right\|_{X},t)\quad \forall x \in \overline{B_r}\  \forall t\geq 0.
\end{equation}

\item[(viii)] {\it globally asymptotically stable at zero uniformly with
            respect to the state} (0-UGAS), if it is 0-UAS and
          \eqref{UniStabAbschaetzung} holds for all $x \in X$.
\end{itemize}
\end{Def}

We stress the difference between the uniform notions 0-UGATT and 0-UGAS and the nonuniform notions 0-GATT and 0-GAS.
For 0-GATT systems, all trajectories converge to the origin,
but their speed of convergence may differ drastically for initial values with the same norm, in contrast to 0-UGATT systems. 
The notions of 0-ULS and 0-UGS are uniform in the sense that there exists an upper bound of the norm of trajectories which is equal for initial states with the same norm.

\begin{remark}
\label{0-GAS_strong_stability}
For ODE systems 0-GAS is equivalent to 0-UGAS, but it is weaker than
0-UGAS in the infinite-dimensional case. For linear systems $\dot{x} =
Ax$, where $A$ generates a strongly continuous semigroup, the
Banach-Steinhaus theorem implies that 0-GAS is equivalent to strong stability of  the associated semigroup $T$ and implies the 0-UGS property.
\end{remark}

For systems that are 0-LIM, trajectories approach the origin arbitrarily closely. Obviously, 0-GATT implies 0-LIM.

\subsection{Stability notions for systems with inputs}

We now consider systems $\Sigma=(X,\Uc,\phi)$ with inputs.
\begin{Def}
System $\Sigma=(X,\Uc,\phi)$ is called
\begin{itemize}%[leftmargin=*]
    \item[(i)] {\it uniformly locally stable (ULS)}, if there exist $ \sigma \in\Kinf$, $\gamma
          \in \Kinf \cup \{0\}$ and $r>0$ such that for all $ x \in \overline{B_r}$ and all $ u
          \in \overline{B_{r,\Uc}}$:
\begin{equation}
\label{GSAbschaetzung}
\left\| \phi(t,x,u) \right\|_X \leq \sigma(\|x\|_X) + \gamma(\|u\|_{\Uc}) \quad \forall t \geq 0.
\end{equation}

  \item[(ii)] {\it uniformly globally stable (UGS)}, if there exist $ \sigma \in\Kinf$, $\gamma
          \in \Kinf \cup \{0\}$ such that for all $ x \in X, u
          \in \Uc$ the estimate \eqref{GSAbschaetzung} holds.
    
    %\item[(iii)] {\it practically uniformly globally stable (pUGS)}, if there exist $ \sigma \in\Kinf$, $\gamma
          %\in \Kinf \cup \{0\}$ and $c>0$ such that for all $ x \in X$,
          %and all $ u \in \Uc$  it holds that
    \item[(iii)] {\it uniformly globally bounded (UGB)}, if there exist $ \sigma \in\Kinf$, $\gamma
          \in \Kinf \cup \{0\}$ and $c>0$ such that for all $ x \in X$,
          and all $ u \in \Uc$  it holds that
\begin{equation}
\label{pGSAbschaetzung}
\left\| \phi(t,x,u) \right\|_X \leq \sigma(\|x\|_X) +\gamma(\|u\|_{\Uc}) + c \quad \forall t \geq 0.
\end{equation}
\end{itemize}

\end{Def}

%\begin{Def}
%System $\Sigma=(X,\Uc,\phi)$ is called
%\begin{itemize}%[leftmargin=*]
    %\item {\it uniformly locally stable (ULS)}, if there exist $\sigma,\gamma\in \Kinf$ and $r>0$ such that 
%\begin{equation}
%\label{GSAbschaetzung}
%\|x\|_X \leq r,\ \|u\|_{\Uc} \leq r,\ t \geq 0 \qrq   \left\| \phi(t,x,u) \right\|_X \leq \sigma(\|x\|_X) +\gamma(\|u\|_{\Uc}).
%\end{equation}
%
  %\item {\it uniformly globally stable (UGS)}, if if there exist $\sigma,\gamma\in \Kinf$
%\begin{equation}
%\label{GSAbschaetzung-2}
%x\in X,\ u \in\Uc,\ t \geq 0 \qrq   \left\| \phi(t,x,u) \right\|_X \leq \sigma(\|x\|_X) +\gamma(\|u\|_{\Uc}).
%\end{equation}
    %
    %\item {\it practically uniformly globally stable (pUGS)}, if there exist $
          %\sigma,\gamma \in \Kinf$ and $c>0$ such that 
%\begin{equation}
%\label{pGSAbschaetzung}
%x\in X,\ u \in\Uc,\ t \geq 0 \qrq  \left\| \phi(t,x,u) \right\|_X \leq \sigma(\|x\|_X) +\gamma(\|u\|_{\Uc}) + c.
%\end{equation}
%\end{itemize}
%
%\end{Def}

\begin{remark}
Trivially, UGB is equivalent to the boundedness property (BND), as defined in \cite[p. 1285]{SoW96}.
%0-pUGS is also known as Lagrange stability of the undisturbed system, see
%e.g. \cite{KaJ11, MiW17a}.
Also, UGB implies BRS, but the converse fails in general.
% as BRS is only a
%property of the finite-time reachable sets.
\end{remark}

\subsection{Attractivity properties for systems with inputs}

We define the attractivity properties for systems with inputs.
\begin{Def}
System $\Sigma=(X,\Uc,\phi)$ has the
\begin{itemize}%[leftmargin=*]
    \item[(i)] {\it asymptotic gain (AG) property}, if there is a $
          \gamma \in \Kinf  \cup \{0\}$ such that for all $\eps >0$, for
          all $x \in X$ and for all $u \in \Uc$ there exists a
          $\tau=\tau(\eps,x,u) < \infty$ such that 
\begin{equation}
\label{AG_Absch}
t \geq \tau\ \quad \Rightarrow \quad \|\phi(t,x,u)\|_X \leq \eps + \gamma(\|u\|_{\Uc}).
\end{equation}

  \item[(ii)] {\it strong asymptotic gain (sAG) property}, if there is a $
    \gamma \in \Kinf \cup \{0\}$ such that for all $x \in X$ and for all
    $\eps >0$ there exists a $\tau=\tau(\eps,x) < \infty $ such that 
for all $u \in \Uc$
\begin{equation}
\label{sAG_Absch}
t \geq \tau \quad \Rightarrow \quad \|\phi(t,x,u)\|_X \leq \eps + \gamma(\|u\|_{\Uc}).
\end{equation}
    
    \item[(iii)] {\it uniform asymptotic gain (UAG) property}, if there
          exists a
          $ \gamma \in \Kinf \cup \{0\}$ such that for all $ \eps, r
          >0$ there is a $ \tau=\tau(\eps,r) < \infty$ such
          that for all $u \in \Uc$ and all $x \in B_{r}$
\begin{equation}    
\label{UAG_Absch}
t \geq \tau \quad \Rightarrow \quad \|\phi(t,x,u)\|_X \leq \eps + \gamma(\|u\|_{\Uc}).
\end{equation}

\end{itemize}

\end{Def}

All three properties AG, sAG and UAG imply that all trajectories converge to the ball of radius $\gamma(\|u\|_{\Uc})$ around the origin as $t \to \infty$. 
The difference between AG, sAG, and UAG is in the kind of dependence of
$\tau$ on the states and inputs.
In UAG systems this time depends (besides $\eps$) only on the norm of the state, in sAG systems, it depends on the state $x$ (and may vary for different states with the same norm), but it does not depend on $u$. In AG systems $\tau$ depends both on $x$ and on $u$.
For systems without inputs, the AG and sAG properties are reduced to
0-GATT and the UAG property becomes 0-UGATT.

Next we define properties, similar to AG, sAG and UAG, which formalize
reachability of the $\eps$-neighborhood of the ball $B_{\gamma(\|u\|_{\Uc})}$ by trajectories of $\Sigma$.
\begin{Def}
We say that $\Sigma=(X,\Uc,\phi)$ has the
\begin{itemize}
    \item[(i)] \textit{limit property (LIM)} if there exists
          $\gamma\in\K\cup\{0\}$ such that 
for all $x\in X$, $u \in \Uc$ and $\eps>0$ there is a $t=t(x,u,\eps)$:
\[
\|\phi(t,x,u)\|_X \leq \eps + \gamma(\|u\|_{\Uc}).
\]
  \item[(ii)] \textit{strong limit property (sLIM)}, if there exists $\gamma\in\K\cup\{0\}$ so that for every $\eps>0$ 
and for every $x\in X$ there exists $\tau = \tau(\eps,x)$ such that 
for all $u\in\Uc$ there is a $t\leq \tau$:
\begin{eqnarray}
\|\phi(t,x,u)\|_X \leq \eps + \gamma(\|u\|_{\Uc}).
\label{eq:sLIM_ISS_section}
\end{eqnarray}
  \item[(iii)] \textit{uniform limit property (ULIM)}, if there exists
    $\gamma\in\K\cup\{0\}$ so that for every $\eps>0$ and for every $r>0$ there
    exists a $\tau = \tau(\eps,r)$ such that 
for all $x$ with $\|x\|_X \leq r$ and all $u\in\Uc$ there is a $t\leq
\tau$ such that 
\begin{eqnarray}
\|\phi(t,x,u)\|_X \leq \eps + \gamma(\|u\|_{\Uc}).
\label{eq:ULIM_ISS_section}
\end{eqnarray}
\end{itemize}
\end{Def}

%\begin{Def}
%We say that $\Sigma$ has a limit property (LIM) if there exists $\gamma\in\K$ so that
%\begin{eqnarray}
%x\in X,\ u \in\Uc \quad\Rightarrow\quad \inf_{t \geq 0} \|\phi(t,x,u)\|_X \leq \gamma(\|u\|_{\Uc}).
%\label{eq:LIM_inf_form_gamma}
%\end{eqnarray}
%\end{Def}

\begin{remark}
\label{rem:LIMAG}
It is easy to see that AG is equivalent to the existence of a
$\gamma\in\Kinf$ for which
\[
\mathop{\overline{\lim}}_{t\to\infty}\|\phi(t,x,u)\|_X\leq\gamma(\|u\|_{\Uc})
\]
and LIM is equivalent to existence of a $\gamma\in\Kinf$ so that
\[
\inf_{t\geq0}\|\phi(t,x,u)\|_X\leq\gamma(\|u\|_{\Uc}),
\]
where in both cases the conditions hold for all $x\in X,u\in {\cal U}$. In
particular, AG implies LIM and on the other hand it is easy to see that
LIM and UGS together imply AG.
\end{remark}

\begin{remark}
For systems without inputs the notions of sLIM and LIM coincide and are
strictly weaker than the ULIM, even for linear infinite-dimensional systems generated by $C_0$-semigroups, see \cite{MiW17a}.
\end{remark}

\subsection{Input-to-state stability}
\label{subsec:ISS}

Now we proceed to the main concept of this paper:
\begin{Def}
\label{Def:ISS}
System $\Sigma=(X,\Uc,\phi)$ is called {\it (uniformly)  input-to-state stable
(ISS)}, if there exist $\beta \in \KL$ and $\gamma \in \K$ 
such that for all $ x \in X$, $ u\in \Uc$ and $ t\geq 0$ it holds that
\begin {equation}
\label{iss_sum}
\| \phi(t,x,u) \|_{X} \leq \beta(\| x \|_{X},t) + \gamma( \|u\|_{\Uc}).
\end{equation}
\end{Def}
The local counterpart of the ISS property is
\begin{Def}
\label{Def:LISS}
System $\Sigma=(X,\Uc,\phi)$ is called {\it (uniformly) locally input-to-state stable
(LISS)}, if there exist $\beta \in \KL$, $\gamma \in \K$ 
and $r>0$ such that the inequality \eqref{iss_sum} holds for all $ x \in
\overline{B_r}$, $u\in \overline{B_{r,\Uc}}$ and $ t\geq 0$.
\end{Def}

Lyapunov functions are a powerful tool for the investigation of ISS and
local ISS. For the class of semilinear systems \eqref{InfiniteDim} they are defined as follows. Let $x \in X$ and $V$ be a
real-valued function defined in a neighborhood of $x$. The Dini derivative 
of $V$ at $x$ corresponding to the input $u$ along the trajectories of $\Sigma$ is defined by
\begin{equation}
\label{ISS_LyapAbleitung}
\dot{V}_u(x)=\mathop{\overline{\lim}} \limits_{t \rightarrow +0} {\frac{1}{t}\big(V(\phi(t,x,u))-V(x)\big) }.
\end{equation}
\begin{Def}
\label{def:LISS}
Let $D\subset X$ be open with $0 \in D$.
A continuous function $V:D \to \R_+$ is called a \textit{LISS Lyapunov
  function} for a system $\Sigma = (X,\phi,\Uc)$,  if there exist $r >0$,
$\psi_1,\psi_2 \in \Kinf$, $\alpha \in \Kinf$ and $\sigma \in \K$ 
such that $\overline{B_r} \subset D$,
\begin{equation}
\label{LyapFunk_1Eig_LISS}
\psi_1(\|x\|_X) \leq V(x) \leq \psi_2(\|x\|_X), \quad \forall x \in \overline{B_r}
\end{equation}
and the Dini derivative of $V$ along the trajectories of $\Sigma$ satisfies
\begin{equation}
\label{DissipationIneq}
\dot{V}_u(x) \leq -\alpha(\|x\|_X) + \sigma(\|u\|_{\Uc})
\end{equation}
for all $x \in \overline{B_r}$ and $u\in \overline{B_{r,\Uc}} $.

\begin{enumerate}[(i)]  
    \item A function $V:X\to\R_+$ is called an ISS Lyapunov function, if 
\eqref{DissipationIneq} holds for all $x\in X, u \in \Uc$.
    \item $V:D\to\R_+$ is called a ($0$-UAS) Lyapunov function, if \eqref{LyapFunk_1Eig_LISS} is satisfied and if \eqref{DissipationIneq} holds for $u\equiv 0$.
\end{enumerate}
\end{Def}

\begin{remark}
     We point out that on the right-hand side of the dissipation inequality
    \eqref{DissipationIneq} the growth bound is given in terms of
    $\|u\|_{\Uc}$ instead of the more familiar $\|u\|_U$ for $u\in U$. 
    For some input spaces this is a necessity, but for
    the input space of bounded piecewise continuous functions, as well as for
    $L^\infty(\R_+,U)$ it is equivalent to require the condition
    \begin{equation*}
        \dot{V}_u(x) \leq -\alpha(\|x\|_X) + \sigma(\|u(0)\|_U)
    \end{equation*}
    for all $x\in X, u\in \Uc$.
    This may be shown similarly to the proof for "implication form"
    Lyapunov functions provided in, \cite[Proposition~5]{DaM13}.
\end{remark}

\section{Motivation and main result}
\label{sec:Motivation_MainResult}

The primary motivation for this manuscript is the following fundamental result due to Sontag and Wang \cite{SoW95}, \cite{SoW96}, which we informally described in the introduction.
\begin{proposition}
\label{Characterizations_ODEs}
% Consider a forward complete system \eqref{eq:ODE_System}, with a nonlinearity $f$, which is Lipschitz continuous on bounded balls, uniformly w.r.t. the second argument, satisfying $f(0,0)=0$.
% Then the combinations of properties depicted in Figure~\ref{EqEigenschaften} are equivalent.
For a forward complete, finite dimensional system \eqref{eq:ODE_System},
the equivalences depicted in Figure~\ref{EqEigenschaften_ODE} hold.
\end{proposition}

%\begin{proposition}
%\label{Characterizations_ODEs}
%Let $X=\R^n$, $U=\R^m$. For an ODE system \eqref{InfiniteDim} the following statements are equivalent:
%\begin{enumerate}
    %\item[(i)] \eqref{InfiniteDim} is ISS
    %\item[(ii)] \eqref{InfiniteDim} is UAG
    %\item[(iii)] \eqref{InfiniteDim} is AG and UGS
    %\item[(iv)] \eqref{InfiniteDim} is AG and 0-ULS
    %\item[(v)] There exist a global smooth ISS-Lyapunov function for \eqref{InfiniteDim}.
%\end{enumerate}
%\end{proposition}
%Actually in \cite{SoW95}, \cite{SoW96} many more restatements have been proved, the most important of which are depicted in Figure~\ref{EqEigenschaften}.

\begin{figure}[tbh]
\centering
\begin{tikzpicture}[>=implies,thick]
\node (ISS) at (2,5) {ISS};
\node (UAG) at (-1,5) {UAG };
\node (AG_UGAS) at (-1,4) {AG\,$\wedge$\,0-UGAS};
\node (AG_LISS) at (-1,3) {AG\,$\wedge$\,LISS};
\node (AG_LS) at (-1,2) {AG\,$\wedge$\,ULS};
\node (AG_GS) at (5,4) {AG\,$\wedge$\,UGS};
\node (LIM_GS) at (5,3) {LIM\,$\wedge$\,UGS};
\node (LIM_LS) at (5,2) {LIM\,$\wedge$\,ULS };
\node (LIM_0LS) at (2,2) {LIM\,$\wedge$\,0-ULS };
%\node (LIM_pGS_LS) at (5,2) {LIM\,$\wedge$\,pUGS\,$\wedge$\,ULS };
%\node (WRS) at   (6,3) {WRS};
%\node (WRSLF) at (6,4) {$\exists$ WRS-LF};
\node (ISSLF) at (5,5) {$\exists$ ISS-LF};

\draw[thick,double equal sign distance,<-] (LIM_LS) to (LIM_0LS);

\draw[thick,double equal sign distance,->] (UAG) to (AG_UGAS);
\draw[thick,double equal sign distance,->] (AG_UGAS) to (AG_LISS);
\draw[thick,double equal sign distance,->] (AG_LISS) to (AG_LS);
\draw[thick,double equal sign distance,->] (AG_LS) to (LIM_0LS);
\draw[thick,double equal sign distance,->] (ISS) to (UAG);
\draw[thick,double equal sign distance,->] (LIM_GS) to (AG_GS);
\draw[thick,double equal sign distance,->] (LIM_LS) to (LIM_GS);
\draw[thick,double equal sign distance,->] (LIM_0LS) to (LIM_LS);
%\draw[thick,double equal sign distance,->] (UAG) to (AG_GS);

%\draw[thick,double equal sign distance,->] (AG_GS) to (WRS);
\draw[thick,double equal sign distance,->] (AG_GS) to (ISSLF);
% \node (UGAS_LF) at (5.1,3) {\textbf{$\slash$}};

%\draw[thick,double equal sign distance,->] (WRS) to (WRSLF);
%\draw[thick,double equal sign distance,->] (WRSLF) to (ISSLF);
\draw[thick,double equal sign distance,->] (ISSLF) to (ISS);

%\draw[thick,->] (LIM_LS) to[bend left=55] (-1.5,2) to[bend left=55] (UAG);
\end{tikzpicture}

\caption{Characterizations of ISS in finite dimensions}
\label{EqEigenschaften_ODE}
\end{figure}
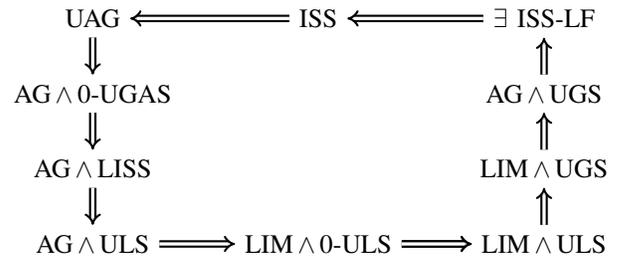

In particular, ISS is not only equivalent to the uniform properties (UAG, the existence of a smooth ISS Lyapunov function), but also to the combination
of the limit property with local stability of the system.

In \cite{Mir16} characterizations of LISS for nonlinear
infinite-dimensional systems of the form \eqref{InfiniteDim} have been studied and the following result \cite[Theorem 4]{Mir16} has been obtained:
\begin{Satz}
\label{Characterization_LISS}
Let Assumption~\ref{Assumption1} hold and assume furthermore
\begin{enumerate}
    \item[$\diamondsuit$] \label{ass:Condition_LISS} there exist $\sigma \in
\K$ and $r >0$ so that for all $v \in B_{r,U}$
and all $x \in \overline{B_r}$ we have
\begin{equation}
\|f(x,v)-f(x,0)\|_X \leq \sigma(\|v\|_U).
\label{eq:Estimate_f_concerning_u}
\end{equation}
\end{enumerate}
Then the following statements are equivalent:
\begin{enumerate}[(i)]
    \item \eqref{InfiniteDim} is 0-UAS.
    \item \eqref{InfiniteDim} has a Lipschitz continuous  0-UAS Lyapunov function.
    \item \eqref{InfiniteDim} has a Lipschitz continuous LISS Lyapunov function.
    \item \eqref{InfiniteDim} is LISS.
\end{enumerate}
\end{Satz}

This result is reminiscent of a classical result on the robustness of
the 0-UAS property \cite[Corollary 4.2.3]{Hen81}. As an easy consequence, we have that for system \eqref{InfiniteDim} 0-UGAS implies LISS, which has already been shown in \cite[Lemma I.1]{SoW96} for ODE systems. 

In the ODE case, the assumption "$\diamondsuit$" in Proposition~\ref{Characterization_LISS} is automatically fulfilled. However, this assumption cannot be dropped for infinite-dimensional systems \eqref{InfiniteDim} as demonstrated by a counterexample in \cite[Section 4]{Mir16}.
We recall another example from \cite[Section 5]{Mir16}:
\begin{examp}
\label{AG_GS_0UGAS_LISS_not_ISS}
Consider a system $\Sigma$ with state space $X=l_1:=\{ (x_k)_{k=1}^{\infty}: \sum_{k=1}^{\infty} |x_k| <\infty  \}$
and input space $\Uc:=PC(\R_+,\R)$.

Let the dynamics of the $k$-th mode of $\Sigma$ be given by 
\begin{eqnarray}
\dot{x}_k(t) = -\frac{1}{1+|u(t)|^k}x_k(t).
\label{eq:CounterEx_AG_UGAS_no_ISS}
\end{eqnarray}
%
%We show here that \eqref{eq:CounterEx_AG_UGAS_no_ISS} is not ULIM, that is for any $\gamma\in\Kinf$ there exist $\eps>0$, $R>0$
%so that for any $\tau$ there exist $x\in B_R$ and $u\in\Uc$: $\|\phi(t,x,u)\|_X> \eps + \gamma(\|u\|_{\Uc})$ for all $t\in[0,\tau]$.
%
%Pick any $\gamma\in\Kinf$, choose any $R$ so that $\gamma^{-1}(\frac{R}{4})>1$ and let $\eps:=\frac{R}{2}$ and $u(r):=\gamma^{-1}(\frac{R}{4})$ for each $r\geq 0$.
%
%Let $\tau>0$ be arbitrary. We have to find $x\in B_R$: 
%\[
%\|\phi(t,x,u)\|_X\geq \frac{3R}{4}.
%\]
%A possible choice is $x:=e_k R$, for $k$ large enough. Here $e_k$ is the $k$-th basis vector of the standard orthonormal basis in $X$. Hence  \eqref{eq:CounterEx_AG_UGAS_no_ISS} is not ULIM.
\hfill$\square$
\end{examp}

According to the analysis in \cite{Mir16} system
\eqref{eq:CounterEx_AG_UGAS_no_ISS} is 0-UGAS, sAG, AG with zero gain, UGS
with zero gain, and LISS with zero gain, but it is not ISS (from the main
result of the present paper it follows that \eqref{eq:CounterEx_AG_UGAS_no_ISS} is not ULIM). 
\textit{This means that all characterizations of ISS in terms of AG or LIM
  together with UGS or 0-UGAS, depicted in Figure~\ref{EqEigenschaften_ODE} are no longer valid for infinite-dimensional systems. This makes the characterization of ISS in infinite dimensions a challenging problem.}

In order to reflect the essential distinctions occurring in these
stability properties, and to obtain a proper generalization of the criteria for ISS, developed by Sontag and Wang in Proposition~\ref{Characterizations_ODEs},
we have introduced several new concepts. These are the uniform and the
strong limit property (ULIM and sLIM), strong input-to-state stability
(sISS) as well as the strong asymptotic gain property (sAG). These notions naturally extend the notions of LIM, AG and UAG introduced in \cite{SoW96}.

The first positive result in characterizations of ISS is the following Lyapunov characterization of ISS, shown in \cite{MiW17c}.
\begin{Satz}
\label{thm:Characterization_ISS}
Assume that $f:X \times U \to X$ is bi-Lipschitz continuous on bounded subsets, that is:
\begin{enumerate}
    \item For all $C>0$ there is $L^1_f(C)>0$, such that
\begin{eqnarray*}
\hspace{-6mm}x, y \in \overline{B_C},\ v\in U \srs \|f(x,v)-f(y,v)\|_X \leq L^1_f(C) \|x-y\|_X.
%\label{eq:Lipschitz-1}
\end{eqnarray*}
    \item  For all $C>0$ there is $L^2_f(C)>0$, such that
\begin{eqnarray*}
\hspace{-6mm}x\in X,\ u,v \in \overline{B_{C,U}} \srs \|f(x,u)-f(x,v)\|_X \leq L^2_f(C) \|u-v\|_U.
%\label{eq:Lipschitz-2}
\end{eqnarray*}
\end{enumerate}
Then \eqref{InfiniteDim} is ISS if and only if there exists a Lipschitz continuous ISS Lyapunov function for \eqref{InfiniteDim}.
\end{Satz}

%Theorem~\ref{thm:Characterization_ISS} gives a proper generalization of one of the characterizations of ISS, depicted in Figure~\ref{EqEigenschaften}.

%\begin{remark}
%In finite dimensions AG $\wedge$ 0-ULS $\Leftrightarrow$ ISS $\Leftrightarrow$ UAG which shows that for ODE systems the difference between AG and UAG is relatively small. The previous example shows that in infinite dimensions this difference is considerably bigger, since even sAG systems with as strong additional properties as UGS and 0-UGAS are still not ISS, and thus not UAG.
%\end{remark}

\subsection{Main result and structure of the paper}
\label{sec:Main_Result_AND_Structure}

\begin{figure*}[tbh]
\centering
\begin{tikzpicture}[>=implies,thick]
%\node (Unif_Level) at (-4,5.5) {Uniform GAS level};
%\node (Nonunif_Level) at (-4.25, 2.5) {Non-uniform GAS level};

\node (ISS) at (-0.5,0.4) {ISS};
\node (UAG) at (-2.7,0.4) {UAG};

%\node (WURS) at  (3.5,-0.5)  {WURS};
%\node (WURSLF) at (7,0.4) {$\exists$ WURS-LF};
\node (ISSLF) at (3,0.4) {\color{blue} $\exists$ ISS-LF};

\draw [rounded corners] (-8.4,1.2) rectangle (-0.2,0.1);
\node (Thm5) at (-4.4,0.9) {\footnotesize Theorem~\ref{thm:UAG_equals_ULIM_plus_LS}};

\node (Thm6) at (-8.75,-0.1) {\footnotesize \color{blue} Theorem~\ref{thm:MainResult_Characterization_ISS_EQ_Banach_Spaces}};
\node (Thm3) at (1.05,0.9) {\footnotesize \color{blue} Theorem~\ref{thm:Characterization_ISS}};

 \node (Thm8) at (-3.2,-1) {\footnotesize Theorem~\ref{wISS_equals_sAG_GS}};

 \node (Thm2) at (4.6,-2.38) {\footnotesize \color{blue} Theorem~\ref{Characterization_LISS}};

\node (Rem) at (-4,-2.35) {\footnotesize Remark~\ref{rem:LIMAG}};

\draw [rounded corners] (-6.4,-1.8) rectangle (-0.2,-0.7);

\draw[thick,double equal sign distance,<->] (ISS) to (UAG);
%\draw[thick,double equal sign distance,->] (ISS) to (WURS);
%\draw[thick,double equal sign distance,->] (WURS) to (WURSLF);
%\draw[thick,double equal sign distance,->] (WURSLF) to (ISSLF);
\draw[thick,double equal sign distance, blue,dashed,<->] (ISSLF) to (ISS);

\node (ULIM_UGS) at (-4.7,0.4) {ULIM$\,\wedge\,$UGS};
\node (ULIM_ULS) at (-7.4,0.4) {ULIM$\,\wedge\,$ULS};
\node (ULIM_0ULS) at (-10.2,0.4) {\color{blue} ULIM$\,\wedge\,$0-ULS};

\draw[thick,double equal sign distance,<->] (ULIM_UGS) to (UAG);
\draw[thick,double equal sign distance,<->] (ULIM_ULS) to (ULIM_UGS);
\draw[thick,double equal sign distance,blue,dashed,<->] (ULIM_0ULS) to (ULIM_ULS);

\node (sISS) at (-0.6,-1.5) {sISS};
\node (sAG_UGS) at (-2.7,-1.5) {sAG$\,\wedge\,$UGS};
\node (sLIM_UGS) at (-5.4,-1.5) {sLIM$\,\wedge\,$UGS};
\draw[<->,double equal sign distance] (sISS) to (sAG_UGS);

\draw[<->,double equal sign distance] (sLIM_UGS) to (sAG_UGS);

\draw[thick,double equal sign distance,->] (-3.3,0) to (-3.3,-0.6);
\draw[red,double equal sign distance,->,degil]  (-2.95,-0.6) to (-2.95,0);
\node[red] (notImp_ix) at (-2.45,-0.3) {\footnotesize(ix)};

\draw[red,double equal sign distance,->,degil]  (0.9,-2.45) to (0.2,-2.45);
\node[red] (vi) at (1.3,-2.45) {\footnotesize(vi)};

\draw[red,double equal sign distance,<-,degil]  (0.9,-2.95) to (0.2,-2.95);
\node[red] (vii) at (1.3,-2.95) {\footnotesize(vii)};

\draw[red,double equal sign distance,->,degil]  (0.9,-1.3) to (0.2,-1.3);
\node[red] (iv) at (1.3,-1.3) {\footnotesize(iv)};

\draw[red,double equal sign distance,<-,degil]  (0.9,-1.7) to (0.2,-1.7);
\node[red] (v) at (1.3,-1.7) {\footnotesize(v)};

%\node[red] (notImp51) at (1,-2.45) { $\not\Leftarrow$\,(vi)};
% \node[red] (notImp52) at (1,-2.95) { $\not\Rightarrow$\,(vii)};
%\node[red] (notImp61) at (1,-1.7) { $\not\Rightarrow$\,(v)};
%\node[red] (notImp62) at (1,-1.3) { $\not\Leftarrow$\,(iv)};

%\node[double equal sign distance,red,rotate=50](NonConv3) at (0.9,-0.6) {$\not\Uparrow$};
%\node[red] (notImp_new) at (0.4,-0.6) {(iii) };

\node (AG_UGS) at (-2.7,-2.75) {AG$\,\wedge\,$UGS};
\node (LIM_UGS) at (-5.4,-2.75) {LIM$\,\wedge\,$UGS};
%\node (LIM_pGS_ULS) at (-5,-3.5) {LIM$\,\wedge\,$pGS$\,\wedge\,$ULS};

\node (AG_ULS) at (-2.7,-3.9) {AG$\,\wedge\,$ULS};
%\node (LIM_ULS) at (-2,-4.5) {LIM$\,\wedge\,$ULS};

\node (AG_0UGAS) at (3.1,-1.5) {AG$\,\wedge\,$0-UGAS};
\node (AG_LISS) at (5.8,-2.75) {\color{blue}AG$\,\wedge\,$LISS};
\node (AG_0UAS) at (3.1,-2.75) {AG$\,\wedge\,$0-UAS};
\node (AG_0LS) at (3.1,-3.9) {AG$\,\wedge\,$0-ULS};
\node (AG_0GAS) at (5.8,-3.9) {AG$\,\wedge\,$0-GAS};

%\node (LIM_0LS) at (3.5,-4.5) {LIM$\,\wedge\,$0-ULS};
%\node (LIM_0GAS) at (7,-4.5) {LIM$\,\wedge\,$0-GAS};

%Relations AG_UGS and AG_ULS
\draw[double equal sign distance,->]  ($(AG_UGS)+(-0.2,-0.3)$) to ($(AG_ULS)+(-0.2,0.3)$);
\draw[red,double equal sign distance,<-,degil]  ($(AG_UGS)+(0.2,-0.3)$) to ($(AG_ULS)+(0.2,0.3)$);
\node[red] (viii) at ($\weight*(AG_ULS)+\weight*(AG_UGS) + (0.75,0)$) {\footnotesize(viii)};

%Relations AG_0UGAS and AG_0UAS
\draw[double equal sign distance,->]  ($(AG_0UGAS)+(-0.2,-0.3)$) to ($(AG_0UAS)+(-0.2,0.3)$);
\draw[red,double equal sign distance,<-,degil]  ($(AG_0UGAS)+(0.2,-0.3)$) to ($(AG_0UAS)+(0.2,0.3)$);
\node[red] (ii) at ($\weight*(AG_0UAS)+\weight*(AG_0UGAS) + (0.6,0)$) {\footnotesize(ii)};

%Relations AG_0UAS and AG_0LS
\draw[double equal sign distance,->]  ($(AG_0UAS)+(-0.2,-0.3)$) to ($(AG_0LS)+(-0.2,0.3)$);
\draw[red,double equal sign distance,<-,degil]  ($(AG_0UAS)+(0.2,-0.3)$) to ($(AG_0LS)+(0.2,0.3)$);
\node[red] (i) at ($\weight*(AG_0LS)+\weight*(AG_0UAS) + (0.6,0)$) {\footnotesize(i)};

% \draw[red,double equal sign distance,<-,degil]  ($(AG_0UAS)+(0.2,-0.3)$) to ($(AG_0LS)+(0.2,0.3)$);
% 
% \node[red] (i) at ($\weight*(AG_0LS)+\weight*(AG_0UAS) + (0.7,0)$) {\footnotesize(i)};

\coordinate (ISS_1) at ($(ISS) + (-0.5,0)$);
\coordinate (AG_0UGAS_1) at ($(AG_0UGAS) + (-0.5,0)$);
\coordinate (coord1) at ($\third*(ISS_1)+\twothirds*(AG_0UGAS_1)$);
\coordinate (coord2) at ($\twothirds*(ISS_1)+\third*(AG_0UGAS_1)$);

\draw[red,double equal sign distance,->,degil]  (coord1) to (coord2);
\node[red] (iii) at ($\weight*(ISS_1) + \weight*(AG_0UGAS_1) + (-0.45,-0.15)$) {\footnotesize(iii)};

\draw[thick,double equal sign distance,->] ($(sAG_UGS)+(-0.2,-0.37)$) to ($(AG_UGS)+(-0.2,0.27)$);

  \coordinate (A2) at ($(sAG_UGS)+(0.2,-0.37)$);
  \coordinate (B2) at ($(AG_UGS)+(0.2,0.27)$);
%\path (B2) -- node (Q2) {\footnotesize\color{red}{???}} (A2);
  %\coordinate (QQ2) at ($(Q2)+(0,-0.2)$);
 %\draw[thick,double equal sign distance,->] (B2) -- (QQ2) -- (A2);
\node[red] (Quest) at (-2.5,-2.23) {\footnotesize ???};
\draw[double equal sign distance,-] (B2) to ($(B2)+(0,0.09)$);
\draw[double equal sign distance,->] ($(B2)+(0,0.42)$) to (A2);

\draw[double equal sign distance,->]  ($(AG_UGS)+(-0.2,-0.3)$) to ($(AG_ULS)+(-0.2,0.3)$);

\draw[thick,double equal sign distance,<->] (AG_UGS) to (LIM_UGS);
%\draw[thick,double equal sign distance,<->] (AG_UGS) to  (LIM_pGS_ULS);

\draw[thick,double equal sign distance,->] (ISS) to (AG_0UGAS);
\draw[thick,double equal sign distance,blue,dashed,<->]  (AG_LISS) to (AG_0UAS);

%\draw[thick,double equal sign distance,->]  (AG_0UAS) to (AG_0LS);
\draw[thick,double equal sign distance,<->]  (AG_0LS) to (AG_0GAS);
%\draw[thick,double equal sign distance,->]  (AG_0LS) to (LIM_0LS);

\draw[thick,double equal sign distance,->] ($(AG_ULS)+(0.9,-0.12)$)  to ($(AG_0LS)+(-1.1,-0.12)$);

  \coordinate (A1) at ($(AG_ULS)+(0.9,0.12)$);
  \coordinate (B1) at ($(AG_0LS)+(-1.1,0.12)$);
\path (B1) -- node (Q) {\footnotesize\color{red}{???}} (A1);
    \draw[thick,double equal sign distance,->] (B1) -- (Q) -- (A1);
%\draw[thick,double equal sign distance,<-] (A1)  -- (B1) node[midway,fill=white] {\footnotesize{\color{red}???}};
%\node[red] (Quest5) at (0,-3.6) {\footnotesize ???};

%\draw[thick,double equal sign distance,->]  (LIM_ULS) to (LIM_0LS);

%\draw[thick,double equal sign distance,<->]  (LIM_0LS) to (LIM_0GAS);
\end{tikzpicture}
%\caption[caption]{Relations between stability notions for infinite-dimensional control systems, satisfying REP and BRS properties:
\caption[caption]{Relations between stability properties of
  infinite-dimensional systems, which have a robust equilibrium point and
  bounded reachability sets:
\begin{itemize}%[leftmargin=5mm]
\setlength{\itemindent}{5mm}
   \item Black arrows show implications or equivalences which hold for general control systems
in infinite dimensions.
\item {\color{red}Red arrows (with the negation sign)}
are implications which do not hold, due to the counterexamples
presented in this paper (see Remark~\ref{rem:Nonimplications}).
\item {\color{blue}Blue dashed equivalences} are proved only for systems of the form \eqref{InfiniteDim} and under certain additional conditions.
\item Black arrows with question marks inside mean that it is not known right now (as far as the authors are concerned), whether the converse implications hold or not.
\end{itemize}
}
\label{ISS_Equiv}
\end{figure*}
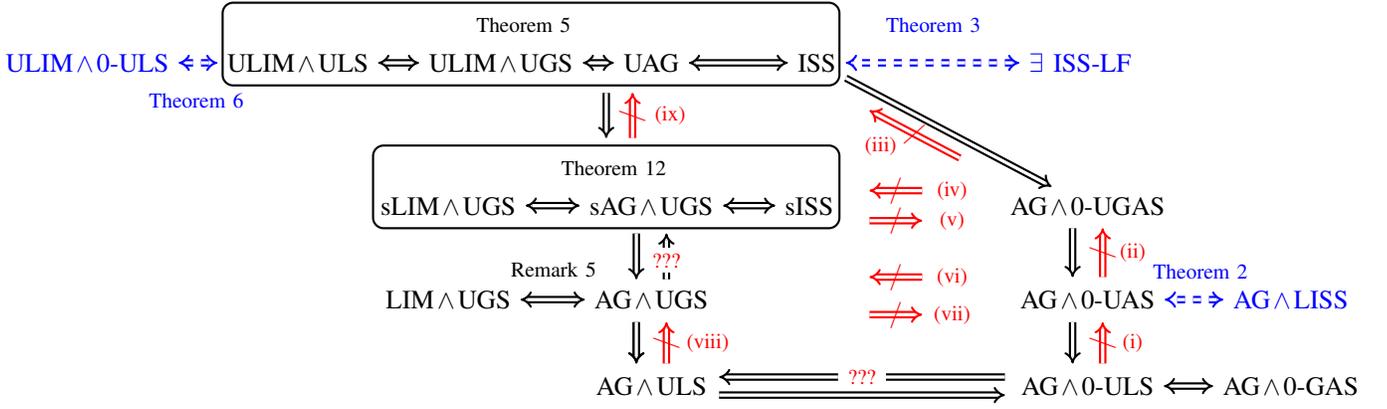

The central result in this paper is the following theorem:
\begin{Satz}
\label{thm:MainResult_Characterization_ISS}
Let $\Sigma=(X,\Uc,\phi)$ be a forward complete system satisfying
the BRS and the CEP property. 
Then the relations depicted in Figure~\ref{ISS_Equiv} hold. 

Black arrows show implications or equivalences which hold for the class of
infinite dimensional systems defined in Definition~\ref{Steurungssystem}; 
blue arrows are valid for semi-linear systems of the class
\eqref{InfiniteDim} under additional assumptions;
the red arrows (with the negation sign) are  implications
which do not hold, due to the counterexamples presented in this paper;
the arrows with question marks indicate that it is not known to us, whether or not these implication hold. 
%
%
%following properties are equivalent:
%\begin{enumerate}[(i)]
    %\item \eqref{InfiniteDim} is ISS.
    %\item \eqref{InfiniteDim} possess a Lipschitz continuous ISS Lyapunov function.
%\end{enumerate}
\end{Satz}

\begin{proof}
Follows from Theorems~\ref{thm:UAG_equals_ULIM_plus_LS},
\ref{wISS_equals_sAG_GS}. The counterexamples for the red arrows are
discussed in Section~\ref{sec:Counterexamples}, see Remark~\ref{rem:Nonimplications}.
\end{proof}

For ODEs all the combinations in Figure~\ref{ISS_Equiv} are
equivalent as for the system~\eqref{eq:ODE_System} we have that AG\,$\wedge$\,0-GAS is
equivalent to ISS by Proposition~\ref{Characterizations_ODEs}.
In contrast, for infinite-dimensional systems, these notions are divided into several groups, which are not equivalent to each other. 

The proof of Theorem~\ref{thm:MainResult_Characterization_ISS} will be given in several steps.

The upper level in Figure~\ref{ISS_Equiv} consists of notions, which are
equivalent to ISS. As in the ODE case, ISS is equivalent to the uniform
asymptotic gain property, and to the existence of a Lipschitz continuous
Lyapunov function. By Example~\ref{AG_GS_0UGAS_LISS_not_ISS}, ISS is not
equivalent to combinations of AG or LIM together with LS or 0-UGAS.
But it turns out, that ISS is equivalent to the combination of ULIM and
LS. This shows that uniformity of attractivity/reachability plays a much
more important role in the infinite-dimensional setting than it does in
the ODE case. The main result in this respect is:
\begin{Satz}
\label{thm:UAG_equals_ULIM_plus_LS}
Let $\Sigma=(X,\Uc,\phi)$ be a forward complete control system. The following statements are equivalent:
\begin{itemize}
    \item[(i)] $\Sigma$ is ISS.
    \item[(ii)] $\Sigma$ is UAG, CEP, and BRS.
    \item[(iii)] \label{cond:ULIM_ULS_BRS_is_ISS} $\Sigma$ is ULIM,
          ULS, and BRS.
    \item[(iv)] $\Sigma$ is ULIM and UGS.
\end{itemize}
\end{Satz}

\begin{proof}
The proof of this result is divided into several lemmas, which will be shown in Section~\ref{ISS_equals_UAG}. Here we show how the result follows from them.

(i) $\Rightarrow$ (ii). 
It is immediate that ISS implies CEP and BRS.
The remaining claim is shown in Lemma~\ref{ISS_implies_UAG}.

(ii) $\Rightarrow$ (iii). Evidently, UAG implies ULIM. By Lemma~\ref{UAG-ULS}\ \ UAG\,$\wedge$\,CEP implies ULS.

(iii) $\Rightarrow$ (iv). Let $\Sigma$ be ULIM and ULS. By
Proposition~\ref{prop:ULIM_plus_mildRFC_implies_pUGS}, ULIM together with
BRS implies UGB. By Lemma~\ref{lem:LS_plus_pUGS_equals_GS}, UGS is
equivalent to  UGB\,$\wedge$\,ULS. 

(iv) $\Rightarrow$ (ii). It is clear that UGS implies CEP and BRS.
The claim follows using Lemma~\ref{lem:ULIM_plus_GS_implies_UAG}.

(ii) $\Rightarrow$ (i). Follows from the  equivalence (ii) $\Iff$ (iv) and Lemma~\ref{lem:UAG_implies_ISS}.
\end{proof}

The next step in the outline of the proof of
Theorem~\ref{thm:MainResult_Characterization_ISS} is as follows.
In Section~\ref{sec:WeakISS} we introduce the new concept of strong input-to-state stability (sISS). 
For nonlinear ODE systems this is equivalent to ISS, see Proposition~\ref{prop:ULIM_equals_LIM_in_finite_dimensions}, and for linear infinite-dimensional systems without inputs sISS is equivalent to strong stability of the associated semigroup $T$ (which justifies the name "strong" for this notion).
We show in Theorem~\ref{wISS_equals_sAG_GS} that
\begin{center}
sISS \quad $\Iff$\quad sAG $\wedge$ UGS \quad $\Iff$\quad sLIM $\wedge$ UGS.
\end{center}

On the other hand, ISS implies the combination AG $\wedge$ 0-UGAS, which
is very different from sISS: sISS does not imply the existence of a uniform
convergence rate for the undisturbed system (and thus, it does not imply
0-UGAS). At the same time AG $\wedge$ 0-UGAS does not ensure the existence of
uniform global bounds for a system with inputs, i.e. UGS is not implied. 

Below the level of sISS and AG $\wedge$ 0-UGAS there are further levels with even weaker properties.
The counterexamples, discussing delicate properties of
infinite-dimensional systems and giving the necessary counterexamples for Figure~\ref{ISS_Equiv}  are discussed in detail in Section~\ref{sec:Counterexamples}.

Finally, we specialize ourselves to the important subclass of
infinite-dimensional systems described by \eqref{InfiniteDim} 
 and discuss
the proof of the blue (dashed) implications. The equivalences labeled by
Theorems~\ref{Characterization_LISS} and~\ref{thm:Characterization_ISS} are cited from the literature and included to
provide a broader picture.
 We show in Section~\ref{sec:Equations_in_Banach_space} that for such
 systems standard assumptions on the nonlinearity $f$ together with the
 BRS property imply continuity of the flow at trivial equilibrium.
This helps to make our main result more precise for systems of the form \eqref{InfiniteDim}.
\begin{Satz}
\label{thm:MainResult_Characterization_ISS_EQ_Banach_Spaces}
Let \eqref{InfiniteDim} satisfy Assumption~\ref{Assumption1} and property
"$\diamondsuit$" from 
Theorem~\ref{Characterization_LISS}. Then the following statements are equivalent.
\begin{itemize}
    \item[(i)] \eqref{InfiniteDim} is ISS.
    \item[(ii)] \eqref{InfiniteDim} is UAG and BRS.
    \item[(iii)] \eqref{InfiniteDim} is ULIM, ULS and BRS.
    \item[(iv)] \eqref{InfiniteDim} is ULIM and UGS.
    \item[(v)] \eqref{InfiniteDim} is ULIM, 0-ULS, and BRS.
\end{itemize}
\end{Satz}

\begin{proof}
    By Lemma~\ref{lem:RobustEquilibriumPoint}, it holds for system
    \eqref{InfiniteDim} that Assumption~\ref{Assumption1} together with BRS
    implies CEP. In conjunction with Theorem~\ref{thm:UAG_equals_ULIM_plus_LS}
    this shows equivalence of  (i) -- (iv).

Clearly, (iii) implies (v). Assume that \eqref{InfiniteDim} is ULIM,
0-ULS, and BRS. 
%Then, according to {\color{blue} Proposition ... in \cite{MiW17a}}, for undisturbed systems it holds that
%\begin{center}
%0-ULIM\,$\wedge$\,0-ULS\,$\wedge$\,BRS \quad$\Iff$\quad 0-UGAS.
%\end{center}
By the BRS property the value
  \begin{equation*}
      \tilde{\beta}(r,t) := \sup \{ \|\phi(t,x,0)\|_X \midset x \in B_r \} 
  \end{equation*}
is finite for all $(r,t) \in \R_+^2$. The function $\tilde{\beta}$ is
increasing in $r$. 
By 0-ULS of \eqref{InfiniteDim},  $\tilde\beta$ is continuous in the first argument at $0$.

Also for fixed $r\geq 0$ we claim that $\lim_{t\to\infty}
 \tilde{\beta}(r,t) = 0$ by ULIM and 0-ULS. To see this let $\sigma$ be
 the function characterizing 0-ULS. Given $\varepsilon>0$ we may by ULIM
 choose a $\tau >0$ such that for all $x\in B_r$ there is a $t\leq \tau$
 with 
 \begin{equation*}
     \|\phi(t,x,0)\|_X \leq \sigma^{-1}(\varepsilon). 
 \end{equation*}
 By 0-ULS and the cocycle property it follows that $\|\phi(t,x,0)\|_X \leq
 \varepsilon$ for all $t\geq \tau$ so that we have the desired
 convergence. We now have that for all $x\in X$ and all $t\geq 0$
\begin{equation*}
     \|\phi(t,x,0)\|_X \leq \max \{ \tilde{\beta}(\|x\|_X,t+s) \midset s \geq
     0 \}    + \|x\|_X e^{-t}. 
 \end{equation*}
This upper bound is a well-defined function of $(\|x\|_X,t)$, continuous w.r.t. the first argument at $\|x\|_X=0$, strictly increasing in $\|x\|_X$ and
strictly decreasing to 0 in $t$. 
%By Proposition~\ref{prop:Upper_estimate_as_KL_function} 
It is easy to see that there is a $\beta \in \mathcal{KL}$ so that
\begin{equation*}
     \|\phi(t,x,0)\|_X \leq \beta(\|x\|_X,t),
 \end{equation*}
and thus \eqref{InfiniteDim} is 0-UGAS.

Since we suppose that the assumption "$\diamondsuit$" of Theorem~\ref{Characterization_LISS} holds, Theorem~\ref{Characterization_LISS} implies LISS (and in particular, ULS) of \eqref{InfiniteDim}. Hence, (v) implies (iii).
\end{proof}

This completes the proof of Theorem~\ref{thm:MainResult_Characterization_ISS}.
In Section~\ref{sec:ODE_LIM_ULIM} we show that our results contain
Proposition~\ref{Characterizations_ODEs} as a special case. This is done
by proving that the notions of LIM and ULIM coincide for ODE systems.
Finally, we specialize our results to the system \eqref{InfiniteDim} without inputs to obtain characterizations of 0-UGAS.

\subsection{ISS via non-coercive ISS Lyapunov functions}

It is well-known that existence of an ISS Lyapunov function implies ISS.
However, construction of an ISS Lyapunov functions for infinite-dimensional systems, especially nonlinear ones, is a challenging task. Already for undisturbed linear systems over Hilbert spaces, "natural" Lyapunov function candidates constructed via solutions of Lyapunov equations are of the form $V(x):= \lel Px,x \rir$, where $\lel\cdot ,\cdot \rir$ is a scalar product in $X$ and $P$ is a linear bounded positive operator which spectrum may contain $0$. Such $V$ fail to satisfy the bound from below in \eqref{LyapFunk_1Eig_LISS}, and possess only a weaker property $V(x)>0$ for $x\neq 0$.
Hence a question appears, whether such "non-coercive" Lyapunov functions still can be used to show ISS of control systems.
A thorough study of this question for uniform global asymptotic stability has been recently performed in \cite{MiW17a}. 
In this section we extend some of the results of this work to the ISS case and show how non-coercive Lyapunov functions can be used to show ISS and ULIM properties of control systems.

\begin{Def}
\label{def:noncoercive_ISS_LF}
A continuous function $V:X \to \R_+$ is called a \textit{noncoercive ISS Lyapunov function} for a system $\Sigma = (X,\phi,\Uc)$,  if there exist $\psi_2,\alpha \in \Kinf$ and $\sigma \in \K$ such that
\begin{equation}
\label{LyapFunk_1Eig_nc_ISS}
0 < V(x) \leq \psi_2(\|x\|_X), \quad \forall x \in X
\end {equation}
and the Dini derivative of $V$ along the trajectories of $\Sigma$ satisfies
\begin{equation}
\label{DissipationIneq_nc}
\dot{V}_u(x) \leq -\alpha(\|x\|_X) + \sigma(\|u\|_{\Uc})
\end{equation}
for all $x \in X$ and $u\in \Uc$.

If \eqref{DissipationIneq_nc} holds just for $u=0$, we call $V$ a \textit{non-coercive Lyapunov function} for the undisturbed system $\Sigma$.
\end{Def}
%
%
%
%In \cite{MiW17a} a notion of a non-coercive Lyapunov function has been introduced
%\begin{Def}
%\label{def:ncLF}
%A continuous function $V:X \to \R_+$ is called a \textit{non-coercive Lyapunov function} for the undisturbed system $\Sigma$,  if there exist $\psi_2 \in \Kinf$, $\alpha \in \Kinf$ and $\sigma \in \K$ such that
%\begin{equation}
%\label{non_coercive_LyapFunk_1Eig}
%0< V(x) \leq \psi_2(\|x\|_X), \quad \forall x \in X \backslash \{0\}
%\end{equation}
%and the Dini derivative of $V$ along the trajectories of the system $\Sigma$ subject to $u\equiv 0$ satisfies
%\begin{equation}
%\label{DissipationIneq_noncoercive}
%\dot{V}_0(x) \leq -\alpha(\|x\|_X)
%\end{equation}
%for all $x \in X$.
%\end{Def}
Note, that in the definition of a non-coercive Lyapunov function we do not
require existence of $\psi_1\in\Kinf$: $\psi_1(\|x\|_X) \leq V(x)$ for all
$x\in X$. 

%{\color{red} 
%Next we have a problem, because in \cite{MiW17a} REP notion has another meaning than in this paper.
%}

We cite the following result for a Lyapunov characterization of $0$-UGAS,
i.e. stability exclusively for the $0$ input. We note that for systems
without inputs the concept
of robust forward completeness and robust equilibrium point used in
\cite{MiW17a} are implied by BRS and CEP.
From the main result in \cite{MiW17a} it follows that:
\begin{proposition}
\label{prop:Non-coerciveLF_Theorem}
Let $\Sigma$ be BRS and CEP. Then $\Sigma$ is 0-UGAS if and only if $\Sigma$ possesses a non-coercive Lyapunov function.

For forward complete systems \eqref{InfiniteDim} satisfying Assumption 1 the same claim holds without assuming CEP.
\end{proposition}

Next we prove two results in this fashion for systems with inputs.
We will use the following lower estimate of $\K$-functions, which is easy to check:
\begin{Lemm}
\label{lem:Kfun_lower_estimates} 
For any $\alpha\in\K$ and any $a,b\geq0$ it holds that 
\begin{eqnarray}
\alpha(a+b)\geq \frac{1}{2}\alpha(a) + \frac{1}{2}\alpha(b).
\label{eq:Kfun_lower_bounds}
\end{eqnarray}
\end{Lemm}

%\begin{proof}
%Let $a\geq b$. Then 
%\begin{eqnarray*}
%\alpha(a+b) \geq \alpha(a) = \frac{1}{2}\alpha(a) + \frac{1}{2}\alpha(a) \geq \frac{1}{2}\alpha(a) + \frac{1}{2}\alpha(b).
%\end{eqnarray*}
%The case $a\leq b$ can be treated analogously.
%\end{proof}

Our first result deals with the ULIM property for general control systems:
\begin{proposition}
\label{prop:ncISS_implies_ULIM} 
Let $\Sigma=(X,\Uc,\phi)$ be a forward complete control system and assume
there exists a non-coercive ISS Lyapunov function for $\Sigma$. Then $\Sigma$ is ULIM.
\end{proposition}

\begin{proof}
Assume that $V$ is a non-coercive ISS Lyapunov function for $\Sigma$ with corresponding $\psi_2,\alpha,\sigma$.
Integrating \eqref{DissipationIneq_nc} from $0$ to $t$, we obtain using \cite[Lemma 3.4]{MiW17a}:
\begin{multline*}
V(\phi(t,x,u))-V(x) \\ \leq -\int_0^t \alpha(\|\phi(s,x,u)\|_X) ds + \int_0^t\sigma(\|u(\cdot + s)\|_{\Uc}) ds.
\end{multline*}
Due to the axiom of shift invariance $\|u(\cdot + s)\|_{\Uc} \leq
\|u\|_{\Uc}$ for $s\geq 0$ and in view of the previous estimate we have
\begin{eqnarray}
\label{ncISS_Eq1}
\int_0^t \alpha(\|\phi(s,x,u)\|_X) ds %&\leq& V(x) - V(\phi(t,x,u)) + t\sigma(\|u\|_{\Uc}) \nonumber\\
                                                                &\leq& V(x) + t\sigma(\|u\|_{\Uc})  \nonumber\\
                                                                &\leq& \psi_2(\|x\|_X) + t\sigma(\|u\|_{\Uc}).
\end{eqnarray}
Now define $\gamma(r):=\alpha^{-1} \big(2\sigma(r)\big)$, $r\in\R_+$ and
$\tau(r,\eps):=2(\psi_2(r)+1)(\alpha(\eps))^{-1}$ for any $r,\eps>0$. 

Assume that $\Sigma$ is not ULIM with these $\gamma$ and $\tau$.
Then there are some $\eps>0$, $r>0$, $x\in \overline{B_r}$ and $u\in\Uc$ so that $\|\phi(t,x,u)\|_X >\eps + \gamma(\|u\|_\Uc)$
for all $t\in[0,\tau(r,\eps)]$.

Using Lemma~\ref{lem:Kfun_lower_estimates} we have for these $\eps,x,u$ and all $t\in[0,\tau(r,\eps)]$ that:
\begin{align*}
\int_0^t \alpha(\|\phi(s,x,u)\|_X) ds &\geq \int_0^t \alpha\big(\eps + \gamma(\|u\|_\Uc)\big) ds \\
& \geq \int_0^t \frac{1}{2}\alpha(\eps) + \sigma(\|u\|_\Uc) ds \\
& =  \frac{t}{2}\alpha(\eps) + t\sigma(\|u\|_\Uc) .
\end{align*}
In particular, for $t:=\tau(r,\eps)$ we obtain that 
\begin{eqnarray}
\label{ncISS_Eq2}
\int_0^{\tau(r,\eps)} \alpha(\|\phi(s,x,u)\|_X) ds \geq \psi_2(r)+1 + \tau(r,\eps)\sigma(\|u\|_\Uc).
\end{eqnarray}

Combining estimates \eqref{ncISS_Eq1} and \eqref{ncISS_Eq2}, we see that 
\begin{eqnarray*}
\psi_2(r)+1 \leq  \psi_2(\|x\|_X) \leq \psi_2(r),
\end{eqnarray*}
a contradiction. This shows that  $\Sigma$ is ULIM.
\end{proof}
Using Theorem~\ref{thm:MainResult_Characterization_ISS_EQ_Banach_Spaces},
Proposition~\ref{prop:ncISS_implies_ULIM} and  recent results in the study of non-coercive Lyapunov functions for nonlinear infinite-dimensional systems \cite{MiW17a},
we are able to prove the following result for system \eqref{InfiniteDim}:

\begin{Satz}
\label{thm:ncISS_LF_sufficient_condition}
Let \eqref{InfiniteDim} satisfy Assumption~\ref{Assumption1} and property "$\diamondsuit$" from Theorem~\ref{Characterization_LISS}.
Let \eqref{InfiniteDim} be BRS and assume there exists a non-coercive ISS Lyapunov function for \eqref{InfiniteDim}.
Then \eqref{InfiniteDim} is ISS.
\end{Satz}

\begin{proof}
\cite[Corollary 4.7]{MiW17a} ensures that \eqref{InfiniteDim} is 0-UGAS and in particular 0-ULS.
Proposition~\ref{prop:ncISS_implies_ULIM} shows that \eqref{InfiniteDim} is ULIM.
Finally, Theorem~\ref{thm:MainResult_Characterization_ISS_EQ_Banach_Spaces}
ensures that \eqref{InfiniteDim} is ISS.
\end{proof}

For a detailed study of non-coercive Lyapunov functions, their advantages and limitations, we refer to \cite{MiW17a}.

\section{Characterizations of ISS}
\label{ISS_equals_UAG}

The technical results needed to prove
Theorem~\ref{thm:UAG_equals_ULIM_plus_LS} will be divided into 2 parts.
First in Section~\ref{subsec:Stability_BRS} we recall restatements of BRS
and ULS. Next in Section~\ref{subsec:Proof_ISS_theorem} we prove our main technical lemmas.
We assume throughout this section that $\Sigma$ is a forward complete system.

\subsection{Stability and boundedness of reachable sets}
\label{subsec:Stability_BRS}

We start with a standard reformulation of the $\varepsilon$-$\delta$ formulations of stability in terms of ${\cal K}$-functions (the proof is straightforward and thus omitted).
\begin{Lemm}
\label{lem:ULS_restatement}
System $\Sigma$ is ULS if and only if for all $\eps>0$ there
exists a $\delta>0$ such that 
\begin{equation}
\label{LS_Restatement}
\|x\|_X \leq\delta,\ \|u\|_{\Uc} \leq \delta,\ t\geq 0 \;\; \Rightarrow\;\;
\|\phi(t,x,u)\|_X \leq\eps.
\end{equation}
\end{Lemm}

%
%\begin{proof}
%"$\Rightarrow$"
%Let $\Sigma$ be ULS. Let $\sigma,\gamma \in \Kinf$ and $r>0$ be such that
%\eqref{GSAbschaetzung} holds for these functions and the neighborhood
%specified by $r$.
%Let $\eps>0$ be arbitrary and choose
%\begin{equation*}
    %\delta=\delta(\eps):=\min\left\{\sigma^{-1}\left(\frac{\eps}{2}\right),
%\gamma^{-1}\left(\frac{\eps}{2}\right), r\right\}.
%\end{equation*}
%With this choice  \eqref{LS_Restatement} follows from \eqref{GSAbschaetzung}.
%
%"$\Leftarrow$" Let \eqref{LS_Restatement} hold. For $\varepsilon \geq 0$ define 
%\begin{multline*}
%\delta(\eps):=\sup\{ s \geq 0 : \|x\|_X \leq s \ \wedge \  \|u\|_{\Uc} \leq s \\
%\Rightarrow \sup_{t \geq 0} \|\phi(t,x,u)\|_X \leq\eps  \}.
%\end{multline*}
%Clearly \eqref{LS_Restatement} implies that $\delta(\cdot)$ is well
%defined, increasing and  continuous in $0$. Let $\hat{\delta} \in {\cal K}$ be any function with
%$\hat{\delta}\leq \delta$ and set $r:= \sup_{s\geq 0}  \hat{\delta}(s) \in
%\R_+ \cup {\infty}$. Define $\gamma:=\hat{\delta}^{-1}:[0,r) \to \R$.
%Then for $\|x\|_X < r$ and $\|u\|_{\Uc}<r$ we have
%\[
%\|\phi(t,x,u)\|_X \leq \gamma(\max\{ \|x\|_X, \|u\|_{\Uc} \}) \leq \gamma(\|x\|_X) + \gamma(\|u\|_{\Uc}),
%\]
%which shows ULS.
%\end{proof}

Lemma~\ref{lem:ULS_restatement} can be interpreted as follows:
the system $\Sigma$ is ULS if and only if $\phi$ is continuous at the equilibrium $0$
% is a robust equilibrium of $\Sigma$ 
and the function $\delta$ in
Definition~\ref{def:RobustEquilibrium_Undisturbed} is independent of
the time $h$.

It is useful to have a restatement of the BRS property in a comparison-functions-like manner.

We call a function $h: \R_+^3 \to \R_+$ increasing, if $(r_1,r_2,r_3) \leq (R_1,R_2,R_3)$
implies that $h(r_1,r_2,r_3) \leq h(R_1,R_2,R_3)$, where we use the component-wise
partial order on $\R_+^3$. We call $h$ strictly increasing if $(r_1,r_2,r_3)
\leq (R_1,R_2,R_3)$ and $(r_1,r_2,r_3) \neq (R_1,R_2,R_3)$ imply $h(r_1,r_2,r_3) <
h(R_1,R_2,R_3)$.

\begin{Lemm}
\label{lem:Boundedness_Reachability_Sets_criterion}
Consider a forward complete system $\Sigma$. The following statements are equivalent:
\begin{enumerate}
    \item[(i)] $\Sigma$ has bounded reachability sets.
    \item[(ii)] There exists a continuous, increasing function $\mu: \R_+^3 \to \R_+$, such that for
all $x\in X, u\in \Uc$ and all $t \geq 0$ we have
 \begin{equation}
    \label{eq:8_ISS}
    \| \phi(t,x,u) \|_X \leq \mu( \|x\|_X,\|u\|_{\Uc},t).
\end{equation}    
    \item[(iii)] There exists a continuous function $\mu: \R_+^3 \to \R_+$ such that for
all $x\in X, u\in \Uc$ and all $t \geq 0$ the inequality \eqref{eq:8_ISS} holds.
\end{enumerate}
\end{Lemm}

The proof is analogous to the proof of \cite[Lemma 2.12]{MiW17a} and
is omitted.

Finally, we recall \cite[Lemma I.2, p. 1285]{SoW96}, which was shown for ODEs, but is proved in the same way for $\Sigma$.
\begin{Lemm}
\label{lem:LS_plus_pUGS_equals_GS}
Consider a forward complete system $\Sigma$. Then
$\Sigma$ is ULS and UGB if and only if it is UGS.
\end{Lemm}

%\begin{proof}
%Assume that $\Sigma$ is ULS and pGS. This means, that there exist $\sigma_1,\gamma_1,\sigma_2,\gamma_2 \in \Kinf$ and $r,c>0$ so that for all $x: \|x\|_X\leq r$ and all $u: \|u\|_{\Uc} \leq r$ it holds that
%\begin{equation*}
%%\label{GSAbschaetzung_tmp}
%\left\| \phi(t,x,u) \right\|_X \leq \sigma_1(\|x\|_X) +\gamma_1(\|u\|_{\Uc})
%\end{equation*}
%and for all $x \in X$ and all $u \in \Uc$ the following estimate holds:
%\begin{equation*}
%%\label{pGSAbschaetzung_tmp}
%\left\| \phi(t,x,u) \right\|_X \leq \sigma_2(\|x\|_X) +\gamma_2(\|u\|_{\Uc}) + c
%\end{equation*}
%
%
%Assume without restriction that $\sigma_2 (s) \geq \sigma_1(s)$ and $\gamma_2 (s) \geq \gamma_1(s)$ for all $s \geq 0$.
%
%Pick $k_1,k_2>0$ so that $c=k_1 \sigma_2(r)$ and $c=k_2 \gamma_2(r)$. Then 
%\[
%c \leq c+c \leq  k_1 \sigma_2(\|x\|_X) + k_2 \gamma_2(\|u\|_{\Uc})
%\]
%for any $x \in X$ and any $u \in \Uc$ so that either $\|x\|_X \geq r$ or $\|u\|_{\Uc} \geq r$.
%Thus for all $x \in X$ and all $u \in \Uc$ it holds that 
%\begin{equation*}
%%\label{pGSAbschaetzung_tmp}
%\left\| \phi(t,x,u) \right\|_X \leq (1+k_1)\sigma_2(\|x\|_X) + (1+k_2)\gamma_2(\|u\|_{\Uc}).
%\end{equation*}
%This shows global stability of $\Sigma$.
%\end{proof}

%{\color{blue} We can exclude this proof - it is the same as in \cite{SoW96} }

\subsection{Proof of Theorem~\ref{thm:UAG_equals_ULIM_plus_LS}}
\label{subsec:Proof_ISS_theorem}

We start with a simple lemma:
\begin{Lemm}
\label{ISS_implies_UAG}
If $\Sigma$ is ISS, then it is UAG.
\end{Lemm}

\begin{proof}
Let $\Sigma$ be ISS with the corresponding $\beta\in\KL$ and $\gamma\in\Kinf$. 
Take arbitrary $\eps, r >0$.  Define $\tau=\tau(\eps,r)$ as
the solution of the equation $\beta(r,\tau)=\eps$ (if it exists,
then it is unique, because of monotonicity of $\beta$ in the second
argument, if it does not exist, we set $\tau(\eps,r):=0$). Then for
all $t \geq \tau$, all $x\in X$ with $\|x\|_X \leq r$ and all $u \in \Uc$ we have
\begin{eqnarray*}
\|\phi(t,x,u)\|_X &\leq& \beta(\|x\|_X,t) + \gamma(\|u\|_{\Uc})  \\
&\leq& \beta(\|x\|_X,\tau) + \gamma(\|u\|_{\Uc}) \\
&\leq& \eps + \gamma(\|u\|_{\Uc}),
\end{eqnarray*}
and the estimate \eqref{UAG_Absch} holds.
\end{proof}

\begin{Lemm}
\label{UAG-ULS}
If $\Sigma$ is UAG and CEP, then it is ULS.
\end{Lemm}

\begin{proof} 
    We will show that \eqref{LS_Restatement} holds so that the claim
    follows from Lemma~\ref{lem:ULS_restatement}. Let $\tau$ and $\gamma$ be the
    functions given by \eqref{UAG_Absch}. Let $\eps>0$ and
    $\tau:=\tau(\eps/2,1)$.
    Pick any $\delta_1>0$ so that
    $\gamma(\delta_1)<\eps/2$. Then for all $x \in X$ with $\|x\|_X \leq 1$
    and all $u \in \Uc$ with $\|u\|_{\Uc} \leq\delta_1$ we have
\begin{equation}\label{eq:2}
\sup_{t\geq \tau}\|\phi(t,x,u)\|_X \leq\frac{\eps}{2}+\gamma(\|u\|_{\Uc})<\eps.
\end{equation}

Since $\Sigma$ is CEP, there is some $\delta_2=\delta_2(\eps,\tau)>0$ so that 
\[
\|\eta\|_X \leq\delta_2 \ \wedge \  \|u\|_{\Uc} \leq \delta_2 \quad \Rightarrow \quad \sup_{t\in[0,\tau]}\|\phi(t,\eta,u)\|_X \leq\eps.
\]
Together with \eqref{eq:2}, this proves  \eqref{LS_Restatement} with $\delta:=\min\lbrace 1,\delta_1,\delta_2 \rbrace$.
\end{proof}

We proceed with
\begin{proposition}
\label{prop:ULIM_plus_mildRFC_implies_pUGS}
Assume that $\Sigma$ is BRS and has the uniform limit property. Then $\Sigma$ is UGB.
\end{proposition}

\begin{proof}
Pick any $r>0$ and set $\eps:=\frac{r}{2}$.
Since $\Sigma$ has the uniform limit property, there exists a $\tau=\tau(r)$ so that 
\begin{eqnarray}
x \in\overline{B_r},\, u\in\Uc \ \Rightarrow \ \exists t\leq \tau:\ \|\phi(t,x,u)\|_X \leq \frac{r}{2} +\gamma(\|u\|_{\Uc}).
\label{eq:ULIM_pGS_ISS_Section1}
\end{eqnarray}
In particular, if $x \in\overline{B_r},\, \|u\|_{\Uc} \leq
\gamma^{-1}(\tfrac{r}{4})$ then there exists a $t\leq \tau$ such that 
\begin{eqnarray}
 \|\phi(t,x,u)\|_X \leq \frac{3r}{4}.
\label{eq:ULIM_pGS_ISS_Section2}
\end{eqnarray}
Without loss of generality we can assume that $\tau$ is increasing in $r$, in particular, it is locally integrable.
Defining $\bar\tau(r):=\frac{1}{r}\int_r^{2r}\tau(s)ds$ we see that $\bar\tau(r) \geq \tau(r)$ and $\bar\tau$ is continuous.

Since $\Sigma$ is BRS, Lemma~\ref{lem:Boundedness_Reachability_Sets_criterion} implies that there exists a continuous, increasing function $\mu: \R_+^3 \to \R_+$, such that for
all $x\in X, u\in \Uc$ and all $t \geq 0$ estimate \eqref{eq:8_ISS} holds. This implies that
\begin{eqnarray}
x\in \overline{B_r},\, \|u\|_{\Uc} \leq \gamma^{-1}(\tfrac{r}{4}),\, t \leq \tau(r) \ \Rightarrow\ \|\phi(t,x,u)\|_X \leq \tilde\sigma(r),
\label{eq:ULIM_pGS_ISS_Section3}
\end{eqnarray}
where $\tilde\sigma:r\mapsto \mu\big(r,\gamma^{-1}(\tfrac{r}{4}),\tau(r)\big)$
is continuous and increasing, since $\mu,\gamma,\tau$ are continuous increasing functions.
  Also from \eqref{eq:ULIM_pGS_ISS_Section2} and \eqref{eq:ULIM_pGS_ISS_Section3} it is clear that $\tilde\sigma(r) \geq \frac{3r}{4}$ for any $r>0$.

Assume that there exist  $x\in \overline{B_r}$, $u\in\Uc$ with $\|u\|_{\Uc} \leq \gamma^{-1}(\tfrac{r}{4})$ and $t\geq0$ so that $\|\phi(t,x,u)\|_X  > \tilde \sigma(r)$. Define 
\[
t_m:=\sup\{s \in[0,t]: \|\phi(s,x,u)\|_X \leq r\}\geq 0.
\] 
The quantity $t_m$ is well-defined, since $\|\phi(0,x,u)\|_X = \|x\|_X
\leq r$ due to the identity property ($\Sigma$\ref{axiom:Identity}).
 
In view of the cocycle property ($\Sigma$\ref{axiom:Cocycle}), it holds that
\[
\phi(t,x,u) = \phi\big(t-t_m,\phi(t_m,x,u),u(\cdot + t_m)\big),
\]
and $u(\cdot + t_m) \in\Uc$, since $\Sigma$ satisfies the axiom of shift invariance. 
Assume that $t-t_m \leq\tau(r)$. Since $\|\phi(t_m,x,u)\|_X\leq r$,
\eqref{eq:ULIM_pGS_ISS_Section3} implies that $\|\phi(t,x,u)\|_X \leq \tilde \sigma(r)$ for all $t \in [t_m,t]$.
Otherwise, if $t-t_m >\tau(r)$, then due to \eqref{eq:ULIM_pGS_ISS_Section2} there exists $t^* < \tau(r)$, so that 
\[
\big\|\phi\big(t^*,\phi(t_m,x,u),u(\cdot + t_m)\big)\big\|_X = \|\phi(t^*+t_m,x,u)\|_X \leq \frac{3r}{4},
\] 
which contradicts the definition of $t_m$, since $t_m+t^*<t$.
Hence
\begin{eqnarray}
x \in\overline{B_r},\, \|u\|_{\Uc} \leq \gamma^{-1}(\tfrac{r}{4}),\, t\geq
0 \ \Rightarrow \ 
  \|\phi(t,x,u)\|_X \leq \tilde\sigma(r).
\label{eq:ULIM_pGS_ISS_Section4}
\end{eqnarray}
Denote $\sigma(r):=\tilde\sigma(r) - \tilde\sigma(0)$, for any $r\geq 0$. Clearly, $\sigma\in\Kinf$.

For each $x\in X$, $u\in\Uc$ define $r:=\max\{\|x\|_X,4\gamma(\|u\|_{\Uc})\}$.
Then \eqref{eq:ULIM_pGS_ISS_Section4} immediately shows for all $x\in X,\ u\in\Uc,\ t\geq 0$ that
\begin{eqnarray*}
\|\phi(t,x,u)\|_X &\leq& \sigma\big(\max\{\|x\|_X,4\gamma(\|u\|_{\Uc})\}\big) + \tilde\sigma(0) \\
&\leq& \sigma(\|x\|_X) + \sigma\big(4\gamma(\|u\|_{\Uc})\big) + \tilde\sigma(0),
\end{eqnarray*}
which shows UGB of $\Sigma$.
\end{proof}

\begin{Lemm}
\label{lem:ULIM_plus_GS_implies_UAG}
If $\Sigma$ is ULIM and UGS, then $\Sigma$ is UAG.
\end{Lemm}

\begin{proof}
Without loss of generality assume that $\gamma$ in the definitions of ULIM
and UGS is the same (otherwise take the maximum of the two).

Pick any $\eps>0$ and any $r>0$. By the uniform limit property, there exists $\gamma\in\Kinf$, independent of $\eps$ and $r$, 
and $\tau=\tau(\eps,r)$ so that for any $x\in \overline{B_r}$,
$u\in \overline{B_{r,\Uc}}$ there exists a $t\leq \tau$ so that
$\|\phi(t,x,u)\|_X \leq \eps + \gamma(\|u\|_{\Uc})$.

In view of the cocycle property, we have from the UGS property that for the above $x,u,t$ and any $s\geq 0$
\begin{eqnarray*}
\|\phi(s+t,x,u)\|_X &=& \|\phi(s,\phi(t,x,u),u(s+\cdot))\|_X \\
            &\leq& \sigma(\|\phi(t,x,u)\|_X) + \gamma(\|u\|_{\Uc})\\
            &\leq& \sigma(\eps + \gamma(\|u\|_{\Uc})) + \gamma(\|u\|_{\Uc}).
\end{eqnarray*}
Now let $\tilde\eps := \sigma(2 \eps)>0$. 
% be so that $\eps:=\frac{1}{2}\sigma^{-1}(\tilde\eps)$.
Using the evident inequality $\sigma(a+b)\leq \sigma(2a) + \sigma(2b)$, valid for any $a,b\geq 0$, we proceed to
\begin{eqnarray*}
\|\phi(s+t,x,u)\|_X \leq \tilde\eps + \tilde\gamma(\|u\|_{\Uc}),
\end{eqnarray*}
where $\tilde\gamma(r) = \sigma(2\gamma(r)) + \gamma(r)$.

Overall, for any $\tilde\eps>0$ and any $r>0$ there exists $\tau=\tau(\eps,r) = \tau(\frac{1}{2}\sigma^{-1}(\tilde\eps),r)$,
so that for $t\geq \tau$ we have
\begin{eqnarray*}
\|\phi(t,x,u)\|_X \leq \tilde\eps + \tilde\gamma(\|u\|_{\Uc}).
\end{eqnarray*}
This shows that $\Sigma$ is UAG.
\end{proof}

The final technical lemma of this section is:
\begin{Lemm}
If $\Sigma$ is UAG and UGS, then $\Sigma$ is ISS.
\label{lem:UAG_implies_ISS}
\end{Lemm}
\begin{proof}
Assume that $\Sigma$ is UAG and UGS and that $\gamma$ in \eqref{GSAbschaetzung} and \eqref{UAG_Absch} are the same (otherwise pick $\gamma$ as a maximum of both of them).
Fix arbitrary $r \in \R_+$.  We are going to
construct a function $\beta \in \KL$ so that \eqref{iss_sum} holds.

From global stability it follows that there exist $\gamma,\sigma \in \Kinf$ such
that for all $ t\geq 0$, all $x \in \overline{B_r}$ and all $ u \in \Uc$ we have
\begin{equation}
    \label{lem7:help}
\|\phi(t,x,u)\|_X \leq \sigma(r) + \gamma(\|u\|_{\Uc}).
\end{equation}
Define $\eps_n:= 2^{-n}  \sigma(r)$, for $n \in \N$. The UAG property implies that there exists a sequence of times
$\tau_n:=\tau(\eps_n,r)$, which we may without loss of generality assume
to be strictly increasing, such that for all $x \in \overline{B_r}$ and all $u \in \Uc$
\[
\|\phi(t,x,u)\|_X \leq \eps_n + \gamma(\|u\|_{\Uc})\quad  \forall t \geq \tau_n.
\]
From \eqref{lem7:help} we see that we may set $\tau_0 := 0$.
Define $\omega(r,\tau_n):=\eps_{n-1}$, for $n \in \N$, $n \neq 0$ and $\omega(r,0):=2\eps_0=2\sigma(r)$.

Now extend the definition of $\omega$ 
%for $t \in \R_+ \backslash \{\tau_n,n \in \N\}$ so that 
to a function
$\omega(r,\cdot) \in \LL$. 
%All such functions satisfy the estimate \eqref{iss_sum}, because for all 
We obtain for $t \in (\tau_n,\tau_{n+1})$, $n=0,1,\ldots$ and $x\in B_r$
%it holds 
that
 $ \|\phi(t,x,u)\|_X \leq \eps_n + \gamma(\|u\|_{\Uc})< \omega(r,t) + \gamma(\|u\|_{\Uc})$.
Doing this for all $r \in \R_+$ we obtain the definition of the function $\omega$.

Now define $\hat \beta(r,t):=\sup_{0 \leq s \leq r}\omega(s,t) \geq
\omega(r,t)$ for $(r,t) \in \R_+^2$. From this definition it follows that, 
for each $t\geq 0$, $\hat\beta(\cdot,t)$ is 
%continuous and $\beta(\cdot,t) \in \Kinf$.
increasing in $r$ and $\hat\beta(r,\cdot)$ is decreasing in $t$ for each $r>0$ as
every $\omega(r,\cdot) \in \LL$.
Moreover, for each fixed $t\geq0$, $\hat \beta(r,t) \leq \sup_{0 \leq s \leq r}\omega(s,0)=2\sigma(r)$, which implies that $\hat\beta$ is continuous in the first argument at $r=0$ for any fixed $t\geq0$.
%Now Proposition~\ref{prop:Upper_estimate_as_KL_function} implies that $(r,t) \mapsto \hat\beta(r,t) + |r|e^{-t}$ may be upper bounded by $\beta\in \KL$ and the estimate 
Now it is easy to see that $(r,t) \mapsto \hat\beta(r,t) + |r|e^{-t}$ may be upper bounded by $\beta\in \KL$ and the estimate 
\eqref{iss_sum}
is satisfied with such a $\beta$. %\hfill $\blacksquare$
\end{proof}

\section{Strong ISS}
\label{sec:WeakISS}

As will be shown in Lemma~\ref{ISS_is_stronger_than_AG_GS}, the
combination of the  AG and UGS properties is weaker than ISS.
Therefore it is natural to ask for a weaker property than ISS which is
equivalent to the combination AG $\wedge$ UGS.
In this section, we prove a partial result of this kind.

%We collect several very preliminary results in this section.

\begin{Def}
\label{Def:sISS}
System $\Sigma$ is called {\it strongly input-to-state stable
(sISS)}, if there exist $\gamma \in \K$, $\sigma \in \Kinf$ and $\beta: X \times \R_+ \to \R_+$, so that

\begin{enumerate}
    \item  $\beta(x,\cdot) \in \LL$ for all $x \in X$, $x \neq 0$
  \item $\beta(x,t) \leq \sigma(\|x\|_X)$ for all $x \in X$ and all $t \geq 0$
    \item for all $ x \in X$, all $u\in \Uc$ and all $t\geq 0$ it
          holds that
\begin {equation}
\label{wISS_sum}
\| \phi(t,x,u) \|_{X} \leq \beta(x,t) + \gamma( \|u\|_{\Uc}).
\end{equation}
\end{enumerate}
\end{Def}

\begin{remark}
Clearly, ISS implies sISS, but the converse implication doesn't hold for infinite-dimensional systems in general.
Due to page limits, we do not give an example of sISS control systems, which are not ISS.
However, for the systems without inputs there are multiple examples showing the difference between GAS and UGAS, already in context of linear PDE systems, see
\cite{Oos00, AvL98} etc.
\end{remark}

In contrast to previous remark, an easy application of Proposition~\ref{Characterizations_ODEs} shows that for ODEs the notions of sISS and ISS are equivalent:
\begin{proposition}
\label{prop:sISS_equals_ISS_for_ODEs}
\eqref{eq:ODE_System} is sISS if and only if \eqref{eq:ODE_System} is ISS.
\end{proposition}

\begin{proof}
ISS trivially implies sISS. Conversely, if \eqref{eq:ODE_System} is sISS, then \eqref{eq:ODE_System} is UGS and AG,
which by   Proposition~\ref{Characterizations_ODEs} implies that \eqref{eq:ODE_System} is ISS.
\end{proof}

Strong ISS can be characterized as follows.
\begin{Satz}
\label{wISS_equals_sAG_GS}
Let $\Sigma=(X,\Uc,\phi)$ be a forward complete control system.
The following statements are equivalent.
\begin{enumerate}[(i)] 
  \item $\Sigma$ is sISS.
  \item $\Sigma$ is sAG and UGS.
  \item $\Sigma$ is sLIM and UGS.
\end{enumerate}
\end{Satz}

\begin{proof}
The equivalence of (i) and (ii) is shown in \cite{MiW16a}; the idea of
proof is quite similar to arguments in the previous section.
%(i) $\Rightarrow$ (ii). 
%Let $\Sigma$ be sISS with the corresponding
%$\beta:X \times \R_+ \to \R_+$ and $\sigma \in \Kinf$ and $\gamma \in
%\mathcal{K}$. By definition, $\Sigma$ is UGS characterized by $\sigma$ and
%$\gamma$.
%
%Fix any $x \in X$ and any $\eps >0$.  Define $\tau=\tau(\eps,x)$ as the
%solution of the equation $\beta(x,\tau)=\eps$ (if this solution exists,
%then it is unique, because of the monotonicity of $\beta$ in the second argument, if it does not exist, we set $\tau(\eps,x)=0$). Then for all $t \geq \tau$ and all $u \in \Uc$ 
%\begin{eqnarray*}
%\|\phi(t,x,u)\|_X &\leq& \beta(x,t) + \gamma(\|u\|_{\Uc})\\
%  &\leq& \beta(x,\tau) + \gamma(\|u\|_{\Uc}) \\
%    &\leq& \eps + \gamma(\|u\|_{\Uc}),
%\end{eqnarray*}
%and the estimate \eqref{sAG_Absch} holds.
%Thus, sISS implies sAG.

(ii) $\Rightarrow$ (iii). This is clear.

(iii) $\Rightarrow$ (ii). Can be proved along the lines of Lemma~\ref{lem:ULIM_plus_GS_implies_UAG}.
%
%(ii) $\Rightarrow$ (i). Assume that $\Sigma$ is UGS and sAG.
%Fix an arbitrary $\delta \in \R_+$ and any $x \in X$ with $\|x\|_X \leq \delta$.  We are going to construct $\beta:X \times \R_+ \to \R_+$ with the properties as in Definition~\ref{Def:sISS}, so that \eqref{wISS_sum} holds.
%
%Uniform global stability of $\Sigma$ implies that there exist $\sigma,\gamma \in\Kinf$ so that for all $t \geq 0$ and for all $u \in \Uc$ it holds that 
%\[
%\|\phi(t,x,u)\|_X \leq \sigma(\delta) + \gamma(\|u\|_{\Uc}).
%\]
%Define $\eps_n:= 2^{-n} \sigma(\delta)$, for $n \in \N$. Due to
% sAG there exists a sequence of times $\tau_n:=\tau(\eps_n,x)$, which we
% assume without loss of generality to be strictly increasing in $n$, such that 
% \[
% \|\phi(t,x,u)\|_X \leq \eps_n + \gamma(\|u\|_{\Uc}) \quad \forall u \in \Uc,\ \forall t \geq \tau_n.
% \]
% Define $\beta(x,\tau_n):=\eps_{n-1}$, for $n \in \N$, $n \neq 0$. 
%%
%
% Now extend the function $\beta(x,\cdot)$ for $t \in \R_+ \backslash
% \{\tau_n,n \in \N\}$ so that $\beta(x,\cdot) \in \LL$ and $\beta(x,t) \leq
% 2 \sigma(\|x\|_X)$ for all $t \geq 0$ (this can be done by choosing the values of $\beta(x,t)$ sufficiently small for $t \in [0,\tau_1)$).
%
% The function $\beta$ satisfies the estimate \eqref{wISS_sum}, because for
% all $t \in (\tau_n,\tau_{n+1})$ it holds that $ \|\phi(t,x,u)\|_X \leq \eps_n + \gamma(\|u\|_{\Uc})< \beta(x,t) + \gamma(\|u\|_{\Uc})$.
% Performing this procedure for all $x \in X$ we obtain the definition of the function $\beta$. 
% This shows $\Sigma$ is sISS.
\end{proof}

\section{Counterexamples}
\label{sec:Counterexamples}

Before we proceed to our main examples, let us take a quick look at linear systems.
\begin{Lemm}
\label{ISS_is_stronger_than_AG_GS}
For linear undisturbed infinite-dimensional systems of the form \eqref{InfiniteDim}, sAG $\wedge$ UGS does not imply LISS. 
\end{Lemm}

\begin{proof}
    Consider the linear system $\dot{x} = Ax$, where $A$ is the generator
    of a $C_0$-semigroup $T(\cdot)$.  For this system, it is observed in
    \cite{MiW15} that ISS is equivalent to 0-UGAS which is, in turn,     equivalent to the exponential stability of the semigroup $T(\cdot)$. By linearity and as there is no input these properties are also equivalent to LISS.

Also, as there is no input sAG is equivalent to AG.
    On the other hand,  using linearity we have the equivalences
    AG $\wedge$ UGS $\Leftrightarrow$ AG $\wedge$ ULS
    $\Leftrightarrow$ 0-GATT $\wedge$ 0-ULS $\Leftrightarrow$ 0-GAS
    $\Leftrightarrow$ strong stability of $T(\cdot)$ (for the last equivalence
    see Remark~\ref{0-GAS_strong_stability}).  Since strong stability of a
    semigroup does not imply exponential stability in general, the claim
    of the lemma follows.
\end{proof}

In this section we construct:
\begin{itemize}
  \item two nonlinear systems $\Sc^1$, $\Sc^3$ without inputs,
    \item two nonlinear systems $\Sc^2$, $\Sc^4$ with inputs,
\end{itemize}
providing counterexamples which show that the following implications are false (note that the axioms ($\Sigma$\ref{axiom:Identity})--($\Sigma$\ref{axiom:Cocycle}) are fulfilled for all $\Sc^i$):
\begin{itemize}
    \item[$\Sc^1$:] (FC)\,$\wedge$\,0-GAS\,$\wedge$\,0-UAS\ \ $\not\Rightarrow$\ \ BRS.
    \item[$\Sc^2$:] (FC)\,$\wedge$\,0-UGAS\,$\wedge$\,AG\,$\wedge$\,LISS\ \ $\not\Rightarrow$\ \ BRS.
    \item[$\Sc^3$:] (FC)\,$\wedge$\,BRS\,$\wedge$\,0-GAS\,$\wedge$\,0-UAS\ \ $\not\Rightarrow$\ \ 0-UGS.
    \item[$\Sc^4$:] (FC)\,$\wedge$\,BRS\,$\wedge$\,0-UGAS\,$\wedge$\,AG\,$\wedge$\,LISS\ \ $\not\Rightarrow$\ \ UGS.
\end{itemize}

System $\Sc^1$ shows that already for undisturbed systems nonuniform global attractivity does not ensure that the solution map $\phi(t,\cdot)$ maps bounded balls into bounded balls. And even if it does, then global stability still cannot be guaranteed, as clarified by system $\Sc^3$. This shows that the difference between nonuniform attractivity and stability is much bigger for nonlinear infinite-dimensional systems than it is for ODEs.

\begin{remark}
\label{rem:Nonimplications}
Before we proceed to  the constructions of the systems $\Sc^i$, let us show how they justify the negated implications depicted in Figure~\ref{ISS_Equiv}.
\begin{itemize}
    \item[(i)] Follows from Lemma~\ref{ISS_is_stronger_than_AG_GS}.
    \item[(ii)] Follows by construction of $\Sc^3$.
    \item[(iii)] Follows by construction of $\Sc^4$.
    \item[(iv)] Follows by construction of $\Sc^4$.
    \item[(v)] Follows from Lemma~\ref{ISS_is_stronger_than_AG_GS}.
    \item[(vi)] Follows by construction of $\Sc^3$ and/or $\Sc^4$.
    \item[(vii)] Follows from Lemma~\ref{ISS_is_stronger_than_AG_GS}.
    \item[(viii)] Follows by construction of $\Sc^4$.
    \item[(ix)] Follows from Lemma~\ref{ISS_is_stronger_than_AG_GS}.
\end{itemize}
In addition, we recall that Example~\ref{AG_GS_0UGAS_LISS_not_ISS}, which
is fully discussed in \cite{Mir16}, shows that 0-UGAS $\wedge$ sAG $\wedge$ AG with zero gain $\wedge$ UGS with zero gain $\wedge$ LISS with zero gain do not imply ISS (and even do not imply ULIM). Hence, the properties of the "second" level (sISS and AG $\wedge$ 0-UGAS) not only are different from each other (in the sense that they do not imply each other), but also even taken together they do not imply ISS.

Finally, systems $\Sc^1$ and $\Sc^2$ show that the systems with
global nonuniform attractivity properties together with very strong
properties near the equilibrium may not even be BRS.
\end{remark}

\begin{examp}[(FC)\,$\wedge$\,0-GAS\,$\wedge$\,0-UAS\ \ $\not\Rightarrow$\ \ BRS]
\label{0-GAS_but_not_GS}
According to Remark~\ref{0-GAS_strong_stability} for \textit{linear}
infinite-dimensional systems 0-GAS implies 0-UGS. Now we show that for nonlinear systems 0-GAS does not even imply BRS of the undisturbed system.
Consider the infinite-dimensional system $\Sc^1$ defined by
\begin{eqnarray}
\Sc^1: \left
\{\begin{array}{l}
\Sc^1_k:
\left
\{\begin{array}{l}
\dot{x}_k = -x_k + x_k^2 y_k - \frac{1}{k^2} x_k^3,\\
\dot{y}_k = - y_k. 
\end{array}
\right. 
\\
k=1,2,\ldots,
\end{array}
\right.
\label{eq:GAS_not_GS}
\end{eqnarray} 
with the state space 
\begin{eqnarray}
X:=l_2= \left\{ (z_k)_{k=1}^{\infty}: \sum_{k=1}^{\infty} |z_k|^2 <\infty , \quad z_k = (x_k,y_k) \in \R^2 \right\}.
\label{eq:l2_space}
\end{eqnarray}
 We show that $\Sc^1$ is forward complete, 0-GAS and 0-UAS but does not have bounded reachability sets.

Before we give a detailed proof of these facts, let us give an informal explanation of this phenomenon.
If we formally place $0$ into the definition of $\Sc^1_k$ instead of the term $- \frac{1}{k^2} x_k^3$, then the state of $\Sc^1_k$ (for each $k$) will exhibit a finite 
escape time, provided $y_k(0)$ is chosen large enough.
The term  $- \frac{1}{k^2} x_k^3$ prevents the solutions of $\Sc^1_k$ from growing to infinity: the solution then looks like a pike, which is then stopped by the damping $- \frac{1}{k^2} x_k^3$, and converges to 0 since $y_k(t)\to0$ as $t\to\infty$.
However, the larger is $k$, the higher will be the peaks, and hence there is no uniform bound for the solution of $\Sc^1$ starting from a bounded ball.
How we proceed to the rigorous proof.

First we argue that $\Sc^1$ is 0-UAS. Indeed, for $r<1$ the Lyapunov function $V(z) = \|z\|^2_{l_2}=\sum_{k=1}^{\infty} (x_k^2+y_k^2 )$ satisfies 
for all $z_k$ with $|x_k| \leq r$ and $|y_k| \leq r$, $k \in \N$, the estimate
\begin{equation}
\label{eq:Sigma1-Vdot}
\begin{aligned}
\dot{V}(z) &= 2 \sum_{k=1}^{\infty} (-x^2_k + x_k^3 y_k - \frac{1}{k^2} x_k^4- y_k^2) \\
&\leq 2 \sum_{k=1}^{\infty} (-x^2_k + |x_k|^3 |y_k| - \frac{1}{k^2} x_k^4- y_k^2)\\
&\leq 2 \sum_{k=1}^{\infty}  ( (r^2 - 1)x^2_k - y_k^2) \\
&\leq 2(r^2 - 1)V(z).
\end{aligned}
\end{equation}
We see that $V$ is an exponential local Lyapunov function for the system \eqref{eq:GAS_not_GS}
and thus \eqref{eq:GAS_not_GS} is locally uniformly asymptotically
stable. Indeed it is not hard to show that the domain of attraction 
contains $\{ z \in l_2 : |x_k| < r, |y_k| < r, \, \forall k \}$.

To show forward completeness and global attractivity of $\Sc^1$ we first point out that every $\Sc^1_k$ is 0-GAS (and hence 0-UGAS, since $\Sc^1_k$ is finite-dimensional). This follows from the fact that any $\Sc^1_k$ is a cascade interconnection of an ISS $x_k$-system (with $y_k$ as an input) and a globally asymptotically stable $y_k$-system, see \cite{Son89}.

Furthermore, for any $z(0) \in l_2$ there exists a finite $N >0$ such that
$|z_k(0)| \leq \tfrac{1}{2}$ for all $k \geq N$. Decompose the norm of $z(t)$ as follows
\[
\|z(t)\|_{l_2} = \sum_{k=1}^{N-1} |z_k(t)|^2 + \sum_{k=N}^{\infty} |z_k(t)|^2.
\]
According to the previous arguments, $\sum_{k=1}^{N-1} |z_k(t)|^2 \to 0$ as $t \to 0$ since all $\Sc^1_k$ are 0-UGAS for $k=1,\ldots,N-1$. 

Since $|z_k(0)| \leq \tfrac{1}{2}$ for all $k \geq N$, we can apply the computations as in \eqref{eq:Sigma1-Vdot}
in order to obtain (for $r:=\frac{1}{2}$) that
\[
\frac{d}{dt}\Big(\sum_{k=N}^{\infty} |z_k(t)|^2\Big) \leq -\frac{3}{2} \sum_{k=N}^{\infty} |z_k(t)|^2.
\]
Hence  
$\sum_{k=N}^{\infty} |z_k(t)|^2$ decays monotonically and exponentially to $0$ as $t \to \infty$.
Overall, $\|z(t)\|_{l_2} \to 0$ as $t \to \infty$ which shows that $\Sc^1$ is forward complete, 0-GAS and 0-UAS.

Finally, we show that $\Sc^1$ is not BRS.
To prove this, it is enough to show that there exists an $r>0$ and $\tau>0$ so that for any $M>0$ there exist $z \in l_2$ and $t \in [0,\tau]$ so that $\|z\|_{l_2} =r$ and 
$\|\phi(t,z,0)\|_{l_2} > M$.

Let us consider $\Sc^1_k$.
For $y_k \geq 1$ and for $x_k \in [0,k]$ it holds that 
\begin{equation}
\dot{x}_k \geq -2x_k + x_k^2.
\label{eq:Counterex_tmp}
\end{equation}
Pick an initial state $x_k(0)=c>0$ (which is independent of $k$) so that the solution of $\dot{x}_k = -2x_k + x_k^2$ blows up to infinity in time $t^*=1$. Now pick $y_k(0)=e$ (Euler's constant) for all $k=1,2,\ldots$. For this initial condition we obtain $y_k(t) = e^{1-t} \geq 1$ for $t \in [0,1]$.
And consequently for $z_k(0)=(c,e)^T$ there exists a time $\tau_k\in
(0,1)$ such that $x_k(\tau_k) = k$ for the solution of  $\Sc^1_k$.

Now consider an initial state $z(0)$ for $\Sc^1$, where 
$z_k(0)=(c,e)^T$ and $z_j(0)= (0,0)^T$ for $j\neq k$.
For this initial state we have that $\|z(t)\|_{l_2} = |z_k(t)|$ and 
\[
\sup_{t \geq 0}\|z(t)\|_{l_2}=\sup_{t \geq 0}|z_k(t)| \geq |x_k(\tau_k)| \geq  k.
\]
As $k\in\N$ was arbitrary, this shows that the system $\Sc^1$ is not BRS. 
\hfill \qedsymbol% $\blacksquare$
\end{examp}

\begin{examp}[(FC)\,$\wedge$\,0-UGAS\,$\wedge$\,AG\,$\wedge$\,LISS\ \ $\not\Rightarrow$\ \ BRS]
\label{0-UGAS_AG_LISS_but_not_GS}
In this modification of Example~\ref{0-GAS_but_not_GS} it is demonstrated that 0-UGAS $\wedge$ AG $\wedge$ LISS does not imply BRS.
Let $\Sc^2$ be defined by
\begin{eqnarray*}
\Sc^2: \left
\{\begin{array}{l}
\Sc^2_k:
\left
\{\begin{array}{l}
\dot{x}_k = -x_k + x_k^2 y_k |u|- \frac{1}{k^2} x_k^3,\\
\dot{y}_k = - y_k. 
\end{array}
\right. 
\\
k=1,2,\ldots,
\end{array}
\right.
%\label{eq:GAS_not_GS}
\end{eqnarray*}
And let the state space of $\Sc^2$ be $l_2$ (see \eqref{eq:l2_space}) and its input space be $\Uc:=PC(\R_+,\R)$.

Evidently, this system is 0-UGAS. Also, it is clear that $\Sc^2$ is not BRS, since for $u \equiv 1$ we obtain exactly the system from Example~\ref{0-GAS_but_not_GS}, which is not BRS. The proof that this system is forward complete, LISS and AG with zero gain mimics the argument we exploited to show 0-GATT of Example~\ref{0-GAS_but_not_GS} and thus we omit it.
\hfill \qedsymbol
\end{examp}

\begin{examp}[(FC)\,$\wedge$\,BRS\,$\wedge$\,0-GAS\,$\wedge$\,0-UAS\ \ $\not\Rightarrow$\ \ 0-UGS]
\label{ex:FC_0GAS_BRS_not_GS}
We construct a counterexample in 3 steps.

\textbf{Step 1.} Let us revisit Example~\ref{0-GAS_but_not_GS} and find
useful estimates from above for the dynamics of the subsystems
$\Sc^1_k$.

We first note that for initial conditions $z^0_k = (x^0_k,y^0_k)$ with $x^0_ky^0_k
\leq 0$ we have for the corresponding solution of $\Sc^1_k$  (see \eqref{eq:GAS_not_GS}) that $|z_k(t)| \leq |z^0_k|$ for all $t\geq0$.

It is easy to check that for each $k\in\N$ and each
$z_k(0)=(x_k(0),y_k(0)) \in \R^2$ with $y_k(0)x_k(0) >0$
the solution of $\Sc^1_k$ for the
initial condition $z_k(0)$ (see \eqref{eq:GAS_not_GS}) can be estimated
in norm by
\begin{eqnarray*}
|x_k(t)| \leq |\hat{x}_k(t)|
\end{eqnarray*}
where $\hat{x}_k(t)$ is the first component of the solution of the system
\begin{eqnarray}
\hat{\Sc}^1_k:
\left
\{\begin{array}{l}
\dot{\hat{x}}_k(t) = \hat{x}_k^2(t) y_k(0) \\
\dot{y}_k = - y_k. 
\end{array}
\right. 
\label{eq:Examp3_HilfSys}
\end{eqnarray}
with initial condition $\hat{z}_k(0) = (\hat{x}_k(0),y_k(0)) =
(x_k(0),y_k(0))$. 
%Assume also that $y_k(0)x_k(0) >0$ 
%(otherwise the dynamics of \eqref{eq:Examp3_HilfSys} is stable, which is a simple case).

This solution of the $\hat{x}_k$-subsystem of \eqref{eq:Examp3_HilfSys} reads as
\[
\hat{x}_k(t) = \frac{x_k(0)}{1-ty_k(0)x_k(0)},
\]
and this solution exists for $t\in[0,\frac{1}{y_k(0)x_k(0)})$.

Now pick any $R>0$ and assume that $z_k(0) = (x_k(0),y_k(0)) \in B_R$.
Since 
\[
\frac{1}{y_k(0)x_k(0)}\geq \frac{2}{y^2_k(0) + x^2_k(0)} \geq \frac{2}{R^2},
\]
the solutions of \eqref{eq:Examp3_HilfSys} for any initial condition $z_k(0)\in B_R$ exist at least on the time interval $[0,2R^{-2})$.
Moreover, for every such solution for 
%$t\in [0,\frac{2}{y_k(0)x_k(0)})$ (and hence for $t\in [0,\frac{4}{R^2}]$) it holds that
$t\in [0,(2y_k(0)x_k(0))^{-1})$ (and in particular for $t\in [0,R^{-2}]$) it holds that
\[
|\hat{x}_k(t)| \leq 2|x_k(0)|.
\]
Overall, for each $R>0$, all $k\in\N$, all $z_k(0)=(x_k(0),y_k(0)) \in
B_R$ and all 
$t\in [0,R^{-2}]$
 the solution of $\Sc^1_k$
corresponding to the initial condition $z_k(0)$ satisfies
\begin{eqnarray}
|z_k(t)| =\sqrt{x_k^2(t) + y_k^2(t)} \leq \sqrt{\hat{x}_k^2(t) + y_k^2(t)} \leq 2 |z_k(0)|.
\label{eq:Examp3_local_estimate_Sigma_Sys}
\end{eqnarray}

\textbf{Step 2.} Now we are going to modify the system $\Sc^1$
by using time transformations.
 Define
$\tilde{x}_k(t):=x_k(\frac{t}{k})$, $\tilde{y}_k(t):=y_k(\frac{t}{k})$,
for any $t\geq 0$ and any $k\geq 1$. In other words, we make the time of the $k$th mode $k$
times slower than the time of $\Sc^1_k$. This new system we denote by
$\tilde\Sc^1$. The equations defining $\tilde\Sc^1$
are
\begin{eqnarray}
\tilde\Sc^1: \left
\{\begin{array}{l}
\tilde\Sc^1_k:
\left
\{\begin{array}{l}
\dot{\tilde{x}}_k = \frac{1}{k}\big(-\tilde{x}_k + \tilde{x}_k^2 \tilde{y}_k - \frac{1}{k^2} \tilde{x}_k^3\big),\\
\dot{\tilde{y}}_k = - \frac{1}{k} \tilde{y}_k. 
\end{array}
\right. 
\\
k=1,2,\ldots
\end{array}
\right.
\label{eq:GAS_BRS_not_GS}
\end{eqnarray} 
Again the state space of $\tilde\Sc^1$ is $l_2$, see \eqref{eq:l2_space}.

We have seen that $\Sc^1$ fails to satisfy the BRS property, since the solutions of subsystems $\Sc^1_k$ at a given time $t$ have larger pikes the larger $k$ is.
A nonuniform change of clocks in $\Sc^1$, performed above, makes such a behavior impossible.
At the same time, $\tilde\Sc^1$ still is not 0-UGS. Next, we show detailed proofs of these facts.

% Since the linear dynamics dominates for small enough states, it is easy to see that there exists $r>0$ so that for all $z(0)\in l_2$: $\|z(0)\|_X\leq r$ it holds that 
From the computation in \eqref{eq:Sigma1-Vdot} it is easy to obtain that
for the dynamics of $\tilde\Sc^1$ we have
$\dot V(z) \leq 0$ if $\|z\|_{l_2}\leq 1$. It
follows that for all $z(0)\in l_2$ with $\|z(0)\|_{l_2}\leq 1$ we have
\[
\|z(t)\|_X\leq \|z(0)\|_X,
\]
and therefore $\tilde\Sc^1$ is 0-ULS.

Forward completeness and global attractivity of $\tilde\Sc^1$ can be shown along the lines of Example~\ref{0-GAS_but_not_GS}.
This shows that $\tilde\Sc^1$ is 0-GAS.

Let us prove that $\tilde\Sc^1$ is BRS.
Pick any $R>0$, any time $\tau>0$ and any $z\in l_2$: $\|z\|_{l_2} \leq R$. In view of \eqref{eq:Examp3_local_estimate_Sigma_Sys}, we have for any $k\in\N$:
\begin{equation}
|\tilde{z}_k(t)| \leq 2 |z_k(0)|\,\quad 
\forall \ t\in [0,kR^{-2}].
\label{eq:Examp3_est}
\end{equation}

Hence there is a $N=N(R,\tau)$ so that the estimate \eqref{eq:Examp3_est}
holds for all $z\in B_R$, all $k\geq N$ and for all $t\in[0,\tau]$.
Thus, for all $z\in B_R$ and all $t\in[0,\tau]$ we have
\begin{eqnarray*}
\|z(t)\|^2_X &=& \sum_{k=1}^{N-1}|z_k(t)|^2 + \sum_{k=N}^{\infty}|z_k(t)|^2 \\
&\leq& \sum_{k=1}^{N-1}|z_k(t)|^2 + 4\sum_{k=N}^{\infty}|z_k(0)|^2 \\
&\leq& \sum_{k=1}^{N-1}|z_k(t)|^2 + 4R^2.
\end{eqnarray*}
Since $N$ is finite and depends on $R$ and $\tau$ only, and since every $z_k$-subsystem is GAS, it is clear that 
\[
\sup \{ \|z(t)\|_{l_2}  \midset \|z(0)\|_{l_2}\leq R,\ t\in[0,\tau]\} <\infty,
\]
so that $\tilde\Sc^1$ is BRS.

In order to show that $\tilde\Sc^1$ is not 0-UGS, recall the construction in Example~\ref{0-GAS_but_not_GS}.
Consider an initial state $z(0)$ for $\tilde\Sc^1$, where 
$\tilde{z}_k(0)=(c,e)^T$ and $\tilde{z}_j(0)= (0,0)^T$ for $j\neq k$.
For this initial state we have that $\|\tilde{z}(t)\|_{l_2} = |\tilde{z}_k(t)| \geq |\tilde{x}_k(t)| = |x_k(\frac{t}{k})|$.
And hence
\[
\sup_{t \geq 0}\|\tilde{z}(t)\|_{l_2} \geq |\tilde{x}_k(k\tau_k)| = |x_k(\tau_k)|  \geq  k.
\]
As $k\in\N$ was arbitrary, this shows that $\tilde\Sc^1$ is not 0-UGS. 

\textbf{Step 3.} Let $c$ be as in Step 2. Let $a:=\min\{c,\tfrac{1}{2}\}$
and choose a smooth function $\xi:\R\to\R$ with
\[
\xi(s):=
\begin{cases}
s & \text{, if } |s| \leq \frac{a}{4}, \\ 
0&\text{, if } |s| > \frac{a}{2}, \\ 
\in [-\frac{a}{2},-\frac{a}{4}]\cup [\frac{a}{4},-\frac{a}{2}],  & \text{, otherwise. }
\end{cases}
\]
Now consider the modification of $\tilde\Sc^1$, which we denote $\Sc^3$.
\begin{eqnarray}
\Sc^3: \left
\{\begin{array}{l}
\Sc^3_k:
\left
\{\begin{array}{l}
\dot{\tilde{x}}_k = -\xi(\tilde{x}_k) +  \frac{1}{k}\big(-\tilde{x}_k + \tilde{x}_k^2 \tilde{y}_k - \frac{1}{k^2} \tilde{x}_k^3\big),\\
\dot{\tilde{y}}_k = -\xi(\tilde{y}_k) - \frac{1}{k} \tilde{y}_k. 
\end{array}
\right. 
\\
k=1,2,\ldots
\end{array}
\right.
%\label{eq:GAS_BRS_not_GS}
\end{eqnarray} 
The additional dynamics generated by $\xi$ improve the stability
properties of $\Sc^3$ when compared to $\tilde\Sc^1$.
In particular, since $\tilde\Sc^1$ is forward complete, 0-GAS, BRS,
$\Sc^3$ also has these properties.
Moreover, in a neighborhood of the origin the dynamics of $\Sc^3_k$ is dominated by the term $-\xi(\tilde{x}_k) = -\tilde{x}_k$ and $-\xi(\tilde{y}_k) = -\tilde{y}_k$ respectively, which renders $\Sc^3$ 0-UAS.
This can be justified e.g. by a Lyapunov argument, as in Example~\ref{0-GAS_but_not_GS}.

Now, since $\xi(s)=0$ for $s > \frac{a}{2}$, the argument, used to show that $\tilde\Sc^1$ is not 0-UGS, shows that $\Sc^3$ is again not 0-UGS.
\hfill \qedsymbol
\end{examp}

\begin{examp}[FC\,$\wedge$\,BRS\,$\wedge$\,0-UGAS\,$\wedge$\,AG\,$\wedge$\,LISS\ \ $\not\Rightarrow$\ \ UGS]
\label{examp:FC_BRS_0UGAS_AG_LISS_not_UGS}
Consider the system $\Sc^4$ with the state space $l_2$ (see \eqref{eq:l2_space}) and its input space be $\Uc:=PC(\R_+,\R)$.
\begin{eqnarray}
\Sc^4: \left
\{\begin{array}{l}
\Sc^4_k:
\left
\{\begin{array}{l}
\dot{\tilde{x}}_k = -\xi(\tilde{x}_k) +  \frac{1}{k}\big(-\tilde{x}_k + \tilde{x}_k^2 \tilde{y}_k |u| - \frac{1}{k^2} \tilde{x}_k^3\big),\\
\dot{\tilde{y}}_k = -\xi(\tilde{y}_k) - \frac{1}{k} \tilde{y}_k. 
\end{array}
\right. 
\\
k=1,2,\ldots
\end{array}
\right.
\label{eq:GAS_BRS_AG_not_UGS}
\end{eqnarray} 
Since this example is a combination of the two previous ones, we omit all details and just mention that
$\Sc^4$ is forward complete, BRS, 0-UGAS, LISS and AG with zero gain, but at the same time $\Sc^4$ is not UGS.
\hfill \qedsymbol
\end{examp}

\section{Robustness of equilibria for differential equations in Banach spaces}
\label{sec:Equations_in_Banach_space}

In this section, we show that for system \eqref{InfiniteDim} satisfying
Assumption~\ref{Assumption1}, boundedness of reachability sets implies the
CEP property. This technical result is needed for the proof of Theorem~\ref{thm:MainResult_Characterization_ISS_EQ_Banach_Spaces}.

\begin{Def}
\label{axiom:Lipschitz}
The flow of \eqref{InfiniteDim} is called Lipschitz continuous on compact intervals (for uniformly bounded inputs), if 
for any $\tau>0$ and any $R>0$ there exists $L>0$ so that for any $x,y \in \overline{B_R}$,
for all $u \in B_{R,\Uc}$ and all $t \in [0,\tau]$ and it holds that 
\begin{eqnarray}
\|\phi(t,x,u) - \phi(t,y,u) \|_X \leq L \|x-y\|_X.
\label{eq:Flow_is_Lipschitz}
\end{eqnarray}    
\end{Def}

We have the following:
\begin{Lemm}
\label{lem:Regularity}
Let  Assumption~\ref{Assumption1} hold 
and assume that \eqref{InfiniteDim}
is BRS.
 Then \eqref{InfiniteDim} has a flow which is Lipschitz continuous on compact intervals for uniformly bounded inputs.
\end{Lemm}

\begin{proof}
Pick any $R>0$, any $x,y \in \overline{B_R}$ and any $u \in B_{R,\Uc}$. 
Let $x(t):=\phi(t,x,u)$, $y(t):=\phi(t,y,u)$ be the solutions of
\eqref{InfiniteDim} defined on the whole nonnegative time axis.

Pick any $\tau >0$ and set
\[
K(R,\tau):= \sup_{\|x\|_X\leq R,\ \|u\|_{\Uc}\leq R,\ t \in [0,\tau]}\|\phi(t,x,u)\|_X,
\]
which is finite since \eqref{InfiniteDim} is BRS. 

Note also that there exist $M,\lambda>0$ such that $\|T(t)\|\leq Me^{\lambda t}$ for all $t\geq 0$.  
We have for any $t \in [0,\tau]$:
\begin{multline*}
\|x(t) - y(t)\|_X  
\leq \|T(t)\|\|x-y\|_X \\
 + \int_0^t{\|T(t-s)\|\|f(x(r),u(r))-f(y(r),u(r))\|_X dr} \\
\leq Me^{\lambda t} \|x-y\|_X + L(K(R,\tau)) \int_0^t Me^{\lambda (t-r)}\|x(r)-y(r)\|_X dr.
\end{multline*}
Define $z_1(t):=e^{-\lambda t}x(t)$, $z_2(t):=e^{-\lambda t}y(t)$. We can rewrite the above implications as
\begin{align*}
\|z_1(t) - z_2(t)\|_X \leq & M \|x-y\|_X \\
&+ M L(K(R,\tau)) \int_0^t \|z_1(r)-z_2(r)\|_X dr.
\end{align*}

According to Gr\"onwall's inequality we obtain for $t\in[0,\tau]$:
\begin{align*}
\|z_1(t) - z_2(t)\|_X  \leq M \|x-y\|_X e^{ M L(K(R,\tau)) t},
\end{align*}
or equivalently
\begin{align*}
\|x(t) - y(t)\|_X  &\leq M \|x-y\|_X e^{ (M L(K(R,\tau)) + \lambda) t}\\
  &\leq M e^{ (M L(K(R,\tau)) + \lambda) \tau} \|x-y\|_X,
\end{align*}
which proves the lemma. 
\end{proof}

Now we show that $x=0,u=0$ is a point of continuity for \eqref{InfiniteDim}.%robust equilibrium of \eqref{InfiniteDim}.
\begin{Lemm}
\label{lem:RobustEquilibriumPoint}
Let  Assumption~\ref{Assumption1} holds and assume that \eqref{InfiniteDim} is BRS. 
Then $\phi$ is continuous at the equilibrium.
%%
%
%$0$ is a robust equilibrium point of \eqref{InfiniteDim}.
\end{Lemm}

\begin{proof}
Pick any $\eps >0$, $\tau \geq 0$, $\delta >0$ and choose any $x \in X$
with $\|x\|_X \leq \delta$ as well as any $u\in B_{\delta,\Uc}$. It holds that
\begin{eqnarray*}
\|\phi(t,x,u) \|_X %&=& \|\phi(t,x,u) -\phi(t,0,u) + \phi(t,0,u)\|_X \\
            &\leq& \|\phi(t,x,u) -\phi(t,0,u)\|_X + \|\phi(t,0,u)\|_X.
\end{eqnarray*}
%Since Assumption~\ref{Assumption1} is fulfilled, 
By Lemma~\ref{lem:Regularity}, the flow of \eqref{InfiniteDim} is
Lipschitz continuous on compact time intervals. Hence there exists a $L(\tau,\delta)$ so that for $t\in[0,\tau]$
\begin{eqnarray*}
\|\phi(t,x,u) -\phi(t,0,u)\|_X \leq L(\tau,\delta)\|x\|_X \leq L(\tau,\delta)\delta.
\end{eqnarray*}
Let us estimate $\|\phi(t,0,u)\|_X$. We have:
\begin{align*}
\|\phi(t,0,u)\|_X =& \Big\|\int_0^t T(t-s)f(\phi(s,0,u),u(s))ds\Big\|_X\\
\leq& \int_0^t \|T(t-s)\| \Big(\big\|f(\phi(s,0,u),u(s))-f(0,u(s))\big\|_X\\
&\qquad\qquad\qquad\qquad + \|f(0,u(s))\|_X\Big)ds.
\end{align*}
Since $f(0,\cdot)$ is continuous, for any $\eps_2>0$ there exists $\delta_2<\delta$ so that $u(s) \in B_{\delta_2}$ implies that 
$\|f(0,u(s))-f(0,0)\|_X \leq \eps_2$. Since $f(0,0)=0$ due to Assumption~\ref{Assumption1}, for the above $u$ we have
$\|f(0,u(s))\|_X \leq \eps_2$.

Due to the BRS property, there exists $K(\tau,\delta_2)$ with
$\|\phi(s,0,u)\|_X\leq K(\tau,\delta_2)$ for any $u\in B_{\delta_2,\Uc}$ and $s\in[0,\tau]$. Now, Lipschitz continuity of $f$ shows that
\begin{equation*}
\|\phi(t,0,u)\|_X \leq  \int_0^t Me^{\lambda (t-s)} \big(L(K(\tau,\delta))\|\phi(s,0,u)\|_X + \eps_2\big)ds.
\end{equation*}
Define $z(t):=e^{-\lambda t}\phi(t,0,u)$, for $t\in[0,\tau]$. We have 
\begin{equation*}
\|z(t)\|_X \leq  M L(K(\tau,\delta_2))\int_0^t \|z(s)\|_Xds + M\tau\eps_2.
\end{equation*}
Now Gr\"onwall Lemma implies that 
\begin{equation*}
\|z(t)\|_X \leq M\tau\eps_2 e^{M L(K(\tau,\delta_2)) t},
\end{equation*}
in other words
\begin{equation*}
\|\phi(t,0,u)\|_X \leq M\tau\eps_2 e^{(M L(K(\tau,\delta_2)) + \lambda) t},
\end{equation*}
Overall, for $x\in B_{\delta_2}$, for $u\in B_{\delta_2,\Uc}$ and for $t\in[0,\tau]$ we have
\begin{equation*}
\|\phi(t,x,u)\|_X \leq L(\tau,\delta_2)\delta_2 + M\tau\eps_2 e^{(M L(K(\tau,\delta_2)) + \lambda) \tau}.
\end{equation*}
To finish the proof choose $\eps_2$ and $\delta_2$ small enough to ensure that 
\[
L(\tau,\delta_2)\delta_2 + M\tau\eps_2 e^{(M L(K(\tau,\delta_2)) + \lambda) \tau}\leq \eps.
\]
\end{proof}

\section{Characterization of ISS for ODEs}
\label{sec:ODE_LIM_ULIM}

This paper introduces the new properties strong ISS (sISS), strong
asymptotic gain and the strong and the uniform limit
property (sLIM and ULIM). It is worth pointing out that this yields no new
concepts for finite dimensional systems of the form \eqref{eq:ODE_System}.
Indeed, for ODE systems Proposition~\ref{Characterizations_ODEs} and Theorem~\ref{thm:MainResult_Characterization_ISS} show the equivalence of sISS and
ISS. The following
Proposition~\ref{prop:ULIM_equals_LIM_in_finite_dimensions} shows that all
versions of the limit property coincide in the finite dimensional
case. For systems without inputs, finite dimensional examples of fixed points that are attractive but not
stable show that sAG does not imply UAG even for systems of the form
\eqref{eq:ODE_System}. At the moment it is not clear whether AG implies
sAG.

\begin{proposition}
\label{prop:ULIM_equals_LIM_in_finite_dimensions}
Assume the finite-dimensional system \eqref{eq:ODE_System} is
forward complete. Then \eqref{eq:ODE_System} is LIM if and only if it is
ULIM.
\end{proposition}

\begin{proof}
    It is clear that ULIM implies LIM.  For the converse statement we will
    make use of \cite[Corollary III.3]{SoW96}. The
    result may be applied as follows. Assume \eqref{eq:ODE_System} is LIM
    and let $\gamma \in {\cal K}_\infty$ be the corresponding gain. Fix
    $\eps>0$, $r>0$ and $R>0$. By the LIM property, for all $x_0 \in \R^n$
    and all $u\in {\cal U}$ with $\| u \|_\infty \leq R$ there is a time
    $t\geq 0$ such that $| \phi(t,x,u)| \leq \tfrac{\eps}{2} + \gamma(R)$. Then
    \cite[Corollary III.3]{SoW96} states that there is a $\tau =
    \tau(\eps,r,R)$ such that for all $x \in \overline{B_r}$, $u\in B_{R,\Uc}$ there exists a
    $t \leq \tau(\eps,r,R)$ such that 
    \begin{equation}
        \label{eq:4}
        | \phi(t,x,u)| \leq \eps + \gamma(R).
    \end{equation}
    With this argument at hand, we now proceed to show ULIM. 
        %Fix $\varepsilon>0$ and $r>0$. 
        We are going to find a $\tilde\tau=\tilde\tau(\eps,r)$ for which the ULIM property holds.
        Define $R_1 := \gamma^{-1}(\max\{r - \eps,0\})$.
        Then for each $u\in\Uc$: $\|u\|_\infty \geq R_1$ and each $x \in \overline{B_r}$ it holds that 
        \begin{equation*}
        | \phi(0,x,u)| = |x| \leq r \leq \eps + \gamma(R_1) \leq \eps + \gamma(\|u\|_\infty),
    \end{equation*}
        and the time $t(\eps,r,u)$ in the definition of ULIM can be chosen for such $u$ as $t:=0$.
        %
        %We may assume that $\eps < r$. Also we only need
    %to show a result for $u\in \mathcal{U}$ such that  $\eps +
    %\gamma(\|u\|_\infty) <r$ as otherwise, trivially,
    %\begin{equation*}
        %| \phi(0,x,u)| = |x| \leq r \leq \eps +
        %\gamma(\|u\|_\infty)\,,\quad \forall \ x \in B_r.
    %\end{equation*}
%
    %Define $R_1 := \gamma^{-1}(r - \eps)$ and 
        
        Now set $\tau_1:= \tau(\frac{\eps}{2},r, R_1)$. Then by the argument leading to \eqref{eq:4} we have for all
    $x\in \overline{B_r}$ and $u\in\Uc: \|u\|_\infty \leq R_1$ a time $t\leq \tau_1$ such    that 
\begin{eqnarray}
            |\phi(t,x,u)| \leq \eps + \gamma(R_1) - \frac{\eps}{2}.
\label{eq:41}
\end{eqnarray}        
Define 
\[
R_2 := \gamma^{-1}\Big(\max\big\{\gamma(R_1)-\frac{\eps}{2},0\big\}\Big) = \gamma^{-1}\Big(\max\big\{r - \frac{3\eps}{2},0\big\}\Big). 
\]
From \eqref{eq:41} we obtain for all $u$ with $R_2 \leq \|u\|_\infty \leq
R_1$ that for the above $t$ we have
$| \phi(t,x,u)| \leq \eps + \gamma(\|u\|_\infty)$. 
        For $k\in\N$ define the times $\tau_k:= \tau(\tfrac{\eps}{2},r,R_k)$ and
\[
R_k := \gamma^{-1}\Big(\max\big\{\gamma(R_{k-1})-\frac{\eps}{2},0\big\}\Big) = \gamma^{-1}\Big(\max\big\{r - \frac{(k+1)\eps}{2},0\big\}\Big). 
\]        
        Repeating the previous argument we see that for all $x\in \overline{B_r}$ and all $u\in\Uc$ with $R_{k+1} \leq \|u\|_\infty \leq R_k$ there is a time
    $t \leq \tau_k$ such that $| \phi(t,x,u)| \leq \eps + \gamma(\|u\|_{\infty})$. 
        %Now set $R_3 := \gamma^{-1}(\gamma(R_2)-\eps/2) = \gamma^{-1}(r -
    %2\eps)$ and $\tau_3 := \tau(\eps/2,r,R_3)$ and repeat the previous
    %argument. In step we obtain the bound \eqref{eq:5} for inputs $u$ of a smaller norm. 
        The procedure ends after finitely many steps because eventually $r - \tfrac{(k+1)\eps}{2}$ becomes negative. 
    The claim now follows for $\tilde\tau := \max \big\{
    \tau_k \;|\; k=1,\ldots, \lfloor \tfrac{2r}{\eps} \rfloor +1 \big\}$, where $\lfloor
    \cdot \rfloor$ denotes the integer part of a real number. 
\end{proof}

\section{Systems without inputs}
\label{No_disturbances}

In this section, we classify the stability notions for abstract systems
$\Sigma=(X,\phi,0)$
without inputs. This simplified picture can be helpful in understanding the general case and at the same time it is interesting in its own right.

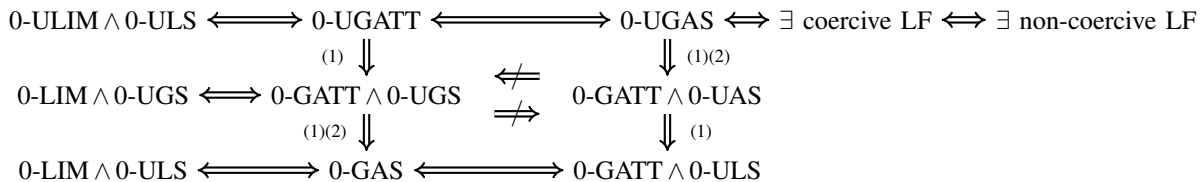
\begin{figure*}[tbh]
\centering
\begin{tikzpicture}[>=implies,thick]

%\node (Unif_Level) at (-4,5.5) {Uniform level};
%\node (Nonunif_Level) at (-4.25, 2.5) {Non-uniform level};

\node (UGAS) at (4,5.5) {0-UGAS};
\node (UGATT) at (0,5.5) {0-UGATT};
\node (ULIM_ULS) at (-3.5,5.5) {0-ULIM\,$\wedge$\,0-ULS};
\node (LF) at (6.5,5.5) {$\exists$ coercive LF};
\node (ncLF) at (9.7,5.5) {$\exists$ non-coercive LF};

\node (UAS_GATT) at (4,4.5) {0-GATT\,$\wedge$\,0-UAS};
\node (GAS_GS) at  (0,4.5) {0-GATT\,$\wedge$\,0-UGS};
\node (LIM_GS) at (-3.5,4.5) {0-LIM\,$\wedge$\,0-UGS};

\node (GAS) at (0,3.5) {0-GAS};
\node (LIM_LS) at (-3.5,3.5) {0-LIM\,$\wedge$\,0-ULS};
%\node (GATT) at (5.5,1.5) {GATT};
\node (GATTLS) at (4,3.5) {0-GATT\,$\wedge$\,0-ULS};
%\node (ATT) at   (5.5,0.5) {ATT};
%\node (AS) at (2.5,1.5) {AS};
%\node (ULS) at (6,1.5) {ULS};

%\node (beta_cont) at (8,2.5) {$\beta$-like est-$C$};
%\node (beta_discont) at (8,1.5) {$\beta$-like est-${\not C}$};

\draw[->,degil,double equal sign distance] (1.7,4.25) to (2.3,4.25);
\draw[<-,degil,double equal sign distance] (1.7,4.75) to (2.3,4.75);

%\node (Sp1) at (2.05,3.75) { $\not\phantom{a}\hspace*{-.5cm}\Longleftarrow$};
%\node (Sp1) at (2.05,4.25) { $\not\phantom{a}\hspace*{-.5cm}\Longrightarrow$};

%\node (S-1) at (4.5,4.75) {{\scriptsize(1)(2)} $\Downarrow$};
\node (S-1) at (4.55,5) {{\scriptsize(1)(2)}};
\draw[->,double equal sign distance] (UGAS) to (UAS_GATT);
%\node (S0) at (0.5,4.75) {{\scriptsize(1)} $\Downarrow$};
\node (S1) at (4.45,4) {\scriptsize(1)};
\draw[->,double equal sign distance] (UAS_GATT) to (GATTLS);

%\node (Fikt1) at (0.5,5.5) {\phantom{a}};
\node (S0) at (-0.45,5) {\scriptsize(1)};
\draw[->,double equal sign distance] (UGATT) to (GAS_GS);

\node (S1.2) at (-0.55,4) {{\scriptsize(1)(2)}};
\draw[->,double equal sign distance] (GAS_GS) to (GAS);

%\node (S2) at (2.5,2) {{\scriptsize(2)}$\cap$};
%\node (S3) at (5.5,2) {{\scriptsize(2)}$\cap$};
%\node (S4) at (5.5,1) {{\scriptsize(2)}$\cap$};
%\node (S5) at (6,2) {$\cap$};

\draw[<->,double equal sign distance](ULIM_ULS) to (UGATT);
%\draw[double,<->] (UGAS) to (UGATT);
\draw[<->,double equal sign distance] (UGAS) to (UGATT);

\draw[<->,double equal sign distance] (LF) to (ncLF);
\draw[<->,double equal sign distance]  (UGAS) to (LF);
%\draw[thick,double,<->] (beta_cont) to (GATTLS);
%\draw[thick,double,<->] (beta_discont) to (GATT);

\draw[<->,double equal sign distance] (GAS_GS) to (LIM_GS);

\draw[<->,double equal sign distance] (GAS) to (GATTLS);
\draw[<->,double equal sign distance] (GAS) to (LIM_LS);

%\draw[thick,->] (B) to[bend left=55] (-1.5,2) to[bend left=55] (A);
\end{tikzpicture}

\caption{Characterizations of 0-UGAS for systems, satisfying BRS and REP properties. Implications marked by (1) resp. (2)
  become equivalences for
  {\scriptsize (1)} ODE systems, see e.g. \cite[Proposition 2.5]{LSW96} and
{\scriptsize (2)} linear systems (as a consequence of the Banach-Steinhaus theorem).}
\label{UGAS_Equiv}
\end{figure*}

\begin{Lemm}
\label{0LIM_0LS_0GAS}
$\Sigma$ is 0-LIM and 0-ULS if and only if $\Sigma$ is 0-GAS.
\end{Lemm}

\begin{proof}
It is clear that 0-GAS implies 0-LIM and 0-ULS. So we only prove the
converse direction.

Pick any $\eps_1>0$. Since $\Sigma$ is 0-ULS, there is a
$\delta_1=\delta_1(\eps_1)>0$ so that  $\|x\|_X \leq \delta_1$ implies
$\|\phi(t,x,0)\|_X \leq \eps_1$ for all $t \geq 0$. 

Pick any $x \in X$. Since $\Sigma$ is 0-LIM, there exists a
$T_1=T_1(x)>0$  such that  $\|\phi(T_1,x,0)\|_X \leq \delta_1$.
By the semigroup property, $\phi(t+T_1,x,0)=\phi(t,\phi(T_1,x,0),0)$ and consequently $\|\phi(t+T_1,x,0)\|_X \leq \eps_1$ for all $t \geq 0$.

Pick a sequence $\{\eps_i\}_{i=1}^{\infty}$ with $\eps_i \to 0$ as $i \to \infty$. According to the above argument, there exists a sequence of times  $T_i=T_i(x)$ such that $\|\phi(t,x,0)\|_X \leq \eps_i$ for all $t \geq T_i$, and thus $\|\phi(t,x,0)\|_X \to 0$ as $t \to \infty$. This shows that $\Sigma$ is 0-GATT, and since we assumed that $\Sigma$ is 0-ULS,  
$\Sigma$ is also 0-GAS.
\end{proof}

Now we can state the main result of this section
\begin{proposition}
\label{Main_Prop_Undisturbed_Systems}
For the system $\Sigma$ without inputs the relations depicted in Figure~\ref{UGAS_Equiv} hold.
\end{proposition}

\begin{proof}
The equivalences on the uniform level follow directly from the equivalence between UAG and ISS, as well as from 
Proposition~\ref{prop:Non-coerciveLF_Theorem}. 
By definition, 0-GAS is equivalent to 0-GATT $\wedge$ 0-ULS, and it is equivalent to 0-LIM $\wedge$ 0-ULS according to Lemma~\ref{0LIM_0LS_0GAS}.

The implications (2) follow since 0-UAS $\Leftrightarrow$ 0-UGAS and 0-ULS $\Leftrightarrow$ 0-UGS for linear systems.
Finally (1) is well-known. 

The observation that 0-UAS $\wedge$ 0-GATT is not implied by and does not imply 0-GAS $\wedge$ 0-UGS follows from Example~\ref{0-GAS_but_not_GS} and since the strong stability of strongly continuous semigroups is weaker than exponential stability.
\end{proof}

\section{Conclusion and relation to previous results}

In this paper we have studied characterizations of ISS properties for a class of infinite-dimensional systems over Banach spaces.

We proved that ISS of infinite-dimensional systems is equivalent to the
uniform asymptotic gain property and to the combination of local stability
with the uniform limit property, introduced here. These results form a proper generalization of well-known characterizations of ISS for systems of ordinary differential equations, proved by Sontag and Wang in \cite{SoW96}.
In contrast to this, we show by means of several counterexamples, that
other characterizations of ISS, known to hold for ODE systems \cite{SoW96}, are no longer valid for infinite-dimensional systems. In particular, combinations of asymptotic or limit properties with uniform global stability are much weaker than ISS.

We introduce the new notion of strong ISS (sISS), which is equivalent to ISS
for nonlinear ODE systems and is equivalent to the strong stability of
$C_0$-semigroups for linear dynamical systems with inputs. In order to
characterize strong ISS, we introduce the notion of strong asymptotic
gain and the strong limit property and prove that the combination of any of these properties with uniform global stability is equivalent to sISS. 

By means of counterexamples, we show the relations between ISS, sISS and
other stability properties, and show that the properties, which were
equivalent to ISS for ODE systems are distinct in the infinite-dimensional world.

In a separate section, we specialize the results of this paper to systems without external inputs and relate these results to the recent characterization of uniform global asymptotic stability by means of non-coercive Lyapunov functions, proved in \cite{MiW17a}.

Finally, using our ISS criteria, we have proved for a broad class of evolution equations in Banach spaces that the existence of a non-coercive ISS Lyapunov function implies ISS.
%, provided the system has bounded reachability sets and some additional minor conditions hold.

A number of questions related to characterizations of strong ISS remain open. In particular it is not known, whether any of following implications hold for nonlinear infinite-dimensional systems: LIM $\Rightarrow$ sLIM, AG $\Rightarrow$ sAG, AG\,$\wedge$\,UGS $\Rightarrow$ sAG\,$\wedge$\,UGS. The answer to these questions will expand considerably our understanding of ISS theory of infinite-dimensional systems.

\bibliographystyle{IEEEtran}
%
%\bibliography{AndersMir}
\bibliography{Mir_LitList-TAC17}
%\bibliography{C:/Arbeiten/Werke/TEX_Data/Mir_LitList}

\begin{IEEEbiography}[{\includegraphics[width=1in,height=1.25in,clip,keepaspectratio]{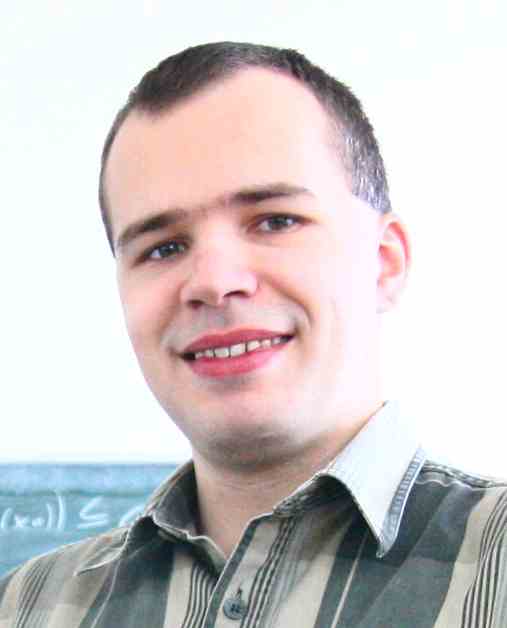}}]{Andrii Mironchenko}
received his MSc at the I.I. Mechnikov Odessa National University in 2008 and his PhD at the University of Bremen in 2012. 
He has held a research position at the University of W\"urzburg and was a Postdoctoral JSPS fellow at the Kyushu Institute of Technology (2013--2014). 
In 2014 he joined the Faculty of Mathematics and Computer Science at the University of Passau. 
His research interests include infinite-dimensional systems, stability theory, hybrid systems and applications of control theory to biological systems. 
\end{IEEEbiography}
%\vfill

\begin{IEEEbiography}[{\includegraphics[width=1in,height=1.25in,clip,keepaspectratio]{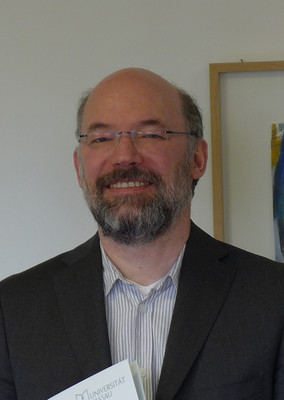}}]{Fabian Wirth}
received his PhD from the Institute of Dynamical Systems at
the University of Bremen in 1995. He has since held positions at the Centre Automatique et Syst{\`e}mes of Ecole des Mines, the
Hamilton Institute at NUI Maynooth, Ireland, the University of W\"urzburg and IBM Research Ireland. 
He now holds the chair for Dynamical Systems at the University
of Passau. His current interests include stability
theory, switched systems and large scale networks with applications to networked systems and in the domain of smart cities.
\end{IEEEbiography}

\end{document}